\documentclass[11pt]{article}
\usepackage{bm, mathrsfs, graphics,amssymb,amsmath,subeqnarray,setspace,graphicx,amsthm,epstopdf,subfigure,color,float,hyperref,authblk}
\usepackage{epstopdf}
\usepackage{enumerate,algorithmic,algorithm}
\usepackage{mathtools}

\usepackage{upref}
\usepackage{svg}
\usepackage{geometry}
\def\md{\mathop{}\mathopen\mathrm{d}}

\newtheorem{proposition}{Proposition}[section]
\newtheorem{remark}{Remark}[section]

\numberwithin{equation}{section}

\title{A structure-preserving collisional particle method for the Landau kinetic equation}

\author[a]{
Kai Du \thanks{Email: kdu@fudan.edu.cn}
}

\author[b]{
Lei Li \thanks{Email: leili2010@sjtu.edu.cn}
}
\author[a]{
Yongle Xie \thanks{Email: 24110840016@m.fudan.edu.cn}
}
\author[b]{
Yang Yu \thanks{Email: yuyang2357@sjtu.edu.cn}
}

\affil[a]{Shanghai Center for Mathematical Sciences, Fudan University, Shanghai 200438, China
}
\affil[b]{School of Mathematical Sciences, Institute of Natural Sciences, MOE-LSC, Shanghai Jiao Tong University, Shanghai 200240, China}

\date{}

\begin{document}

\maketitle

\begin{abstract}
In this paper, we propose and implement a structure-preserving stochastic particle method for the Landau equation. The method is based on a particle system for the Landau equation, where pairwise grazing collisions are modeled as diffusion processes. By exploiting the unique structure of the particle system and a spherical Brownian motion sampling, 
the method avoids additional temporal discretization of the particle system, ensuring that the discrete-time particle distributions exactly match their continuous-time counterparts. 
The method achieves $O(N)$ complexity per time step and preserves fundamental physical properties, including the conservation of mass, momentum and energy, as well as entropy dissipation. It demonstrates strong long-time accuracy and stability in numerical experiments.
Furthermore, we also apply the method to the spatially non-homogeneous equations through a case study of the Vlasov--Poisson--Landau equation.
\end{abstract}

\section{Introduction}
The Landau equation is a fundamental kinetic equation that describes the evolution of the distribution of charged particles in a collisional plasma where grazing collisions are predominant \cite{landau1961kinetic,villani2002review}.
In the space-homogeneous case, the Landau equation governs the evolution of the density $f(t,v)$ of particles with velocity $v$ at time $t$:
\begin{equation}\label{spatially homogeneous Landau equation}
    \partial_t f=
    Q(f,f):=\frac{1}{2}\nabla_{v}\cdot\bigg( \int_{\mathbb{R}^d}
    A(v-v_{\ast})(f(v_{\ast})\nabla_{v}f(v)-f(v)\nabla_{v_\ast}f(v_\ast))\md v_\ast\bigg).
\end{equation}
The collision kernel $A$ is given by
\begin{equation*}
    A(z)=\Lambda|z|^{\gamma}\big(|z|^{2}I_d-z\otimes z\big),
\end{equation*}
where $\Lambda>0$ is the collision strength, $I_d$ is the identity matrix,
and the parameter $\gamma$ can take values within the range $(-d-1,1]$.
The most important case is when $d=3$ and $\gamma=-3$ that corresponds to the original Landau equation for charged particles interacting with Coulomb potentials, derived from the Boltzmann equation in the grazing collision limit.

The well-posedness of the Landau equation has been extensively studied in the literature (see e.g. \cite{villani1998spatially,desvillettes2000spatially,fournier2009well,fournier2021stability} and references therein), covering various aspects of existence, uniqueness, and regularity under different conditions. 
Recently, Guillen and Silvestre \cite{guillen2023landau} achieved a significant breakthrough by proving the global well-posedness of the Landau equation with smooth initial data. Their paper also provides a detailed overview of the historical developments and key references in this area.

Various numerical methods have been developed for solving the Landau equation, including finite difference schemes \cite{whitney1970finite}, the Fourier--Galerkin spectral method \cite{pareschi2000fast}, the direct simulation Monte Carlo method \cite{bird1976molecular,bobylev2000theory,YAN201618,lei2024adaptivesamplingaccelerateshybrid}, and particle methods \cite{fournier2009particle,carrillo2020particle,carrillo2021random}. For comprehensive reviews on this topic, we refer readers to \cite{buet2001comparison,dimarco2015numerical,carrillo2020particle, bobylev2024}. 
Among these, particle methods have seen significant advancements over the past few decades. 
These methods approximate the solution as a linear combination of Dirac $\delta$-functions centered at particle locations.

Particle methods are categorized into deterministic and stochastic approaches. 
In deterministic particle methods, particle locations and weights evolve based on ODE systems derived from the weak formulation of the target equation. 
Carrillo et al. \cite{carrillo2020particle} proposed a gradient-flow-based deterministic particle method that preserves critical physical properties, such as conservation of mass, momentum, energy, and entropy dissipation.
The evaluation of the collision operator \( Q(f, f) \) in this method typically requires \( O(N^2) \) operations, where \( N \) is the number of velocity points.
With the help of treecode summation, it can be reduced to \( O(N \log N) \).
More recently, the random batch particle method \cite{carrillo2021random} further improved efficiency, achieving \( O(N) \) complexity.
However, these are primarily numerical approximations, which effectively recover density but do not fully capture individual particle trajectories.

A typical stochastic particle approximation for the Landau equation stems from its mean-field SDE formulation, which is thought to capture the inherent randomness of particle trajectories.  
Formally, the Landau equation corresponds to the McKean–Vlasov SDE
\begin{equation}\label{MV SDE}
    \mathrm{d}V = K \ast f(V) \, \mathrm{d}t + \sqrt{A \ast f}(V) \, \mathrm{d}W,
\end{equation}
where \( K := \nabla \cdot A =(1-d)\Lambda |z|^{\gamma}z \), and \( f \) represents both the density function of \( V \) and the solution to the Landau equation. 
A direct particle approximation of this SDE \cite{fournier2009particle} takes the form
\begin{equation}\label{direct approximation of MV eq}
    \mathrm{d}V_i = \frac{1}{N-1} \sum_{j \neq i} K(V_i - V_j) \, \mathrm{d}t 
    + \sqrt{\frac{1}{N-1} \sum_{j \neq i} A(V_i - V_j)} \, \mathrm{d}W_i, \quad i = 1, \ldots, N,
\end{equation}
Convergence results for this system are established in specific cases, as shown in \cite{fournier2016propagation, fournier2009well, carrillo2024entropy} in some cases. 
The computational cost of simulating the particle system \eqref{direct approximation of MV eq} \( O(N^2) \) per time step, which poses challenges for large-scale applications.

In this paper, we develop a stochastic particle method for the Landau equation based on our previous work \cite{du2024collision}. 
The method is built upon two key principles: collisions occur pairwise, and the grazing collision between two particles can be approximate by a diffusion process.
Specifically, in each collision time window (of width $\Delta t$), particles are randomly paired, and the motion of each particle is governed by the diffusion process:
\begin{equation}\label{collisionPS1}
    \md V_i = K(V_i-V_{\theta(i)})\md t + \sqrt{A}(V_i-V_{\theta(i)})\md W_i, \quad i=1,\cdots,N,
\end{equation}
where $V_{\theta(i)}$ denotes the paired particle of $V_i$.
In \cite{du2024collision}, we introduced this particle system and proved that, under regularity conditions on the coefficients, its empirical distribution converges to the solution of the Landau equation, with the convergence rate quantified using relative entropy.
This system can be formally interpreted as a random batch method for the 
particle system \eqref{direct approximation of MV eq} with a batch size of $p=2$ \cite{jin2020random}, 
but here the random batch structure arises naturally from physical collisions rather than as a purely approximation. We remark that the random batch approximation with any batch size $p>2$ is not a consistent approximation of McKean--Vlasov SDE~\eqref{MV SDE} unless $p\to\infty$ as $N\to\infty$. The mechanism of the approximation in \eqref{collisionPS1} is intrinsically different from \eqref{direct approximation of MV eq}.

This paper focuses on the numerical performance of the particle system, but with the particular choice of Brownian motions
\[
W_{\theta(i)}=-W_i.
\]
A key structure of the system with this choice of Brownian motions is that the relative velocity is a spherical Brownian motion. By exploiting the unique structure of the system and utilizing spherical Brownian motion sampling, we construct an exact numerical scheme that preserves the particle distribution, ensuring consistency with the continuous system (see Algorithm~\ref{Sphere BM}).  
Overall, the proposed method offers several notable advantages:
\begin{itemize}
\item \emph{Structure Preservation}. The method preserves important physical properties of the Landau equation, such as the conservation of mass, momentum, and energy, as well as the dissipation of entropy, even in its temporal discretization implementation (see Section \ref{sec:discretization}).
These properties ensure that the numerical system stays consistent with the underlying physics, making it reliable for long-time simulations. 
Numerical experiments later in the paper show its stability and accuracy over extended time periods.

\item \emph{Accurate Particle Properties}. By modeling collisions directly as pairwise interactions in the grazing regime, the method provides a detailed representation of individual particle properties, 
complementing approaches that focus on approximating the overall particle distribution. 
This helps achieve a more comprehensive understanding of particle dynamics.

\item \emph{Computational Efficiency}. The computational cost per time step is $O(N)$, making the method highly efficient for large systems.
Additionally, the grouping structure of the algorithm is naturally well-suited for parallel computation, further enhancing its scalability and efficiency.

\item \emph{Improved Sampling}. Unlike traditional direct simulation Monte Carlo (DSMC) methods, which simulate Boltzmann collisions in the grazing regime to approximate Landau collisions, our approach directly simulates the Landau collision. This eliminates an extra layer of approximation and avoids the need for rejection sampling when updating velocities, reducing computational cost and making the method simpler to implement.
\end{itemize}
This method thus provides a practical and efficient tool for solving the Landau equation, offering both physical accuracy and computational efficiency.

The rest of the paper is organized as follows. 
Section~2 introduces the collisional particle system and demonstrates that it preserves the fundamental physical properties of the Landau equation. 
Section~3 describes the construction of the exact temporal discretization scheme and discusses its distinctive features. 
As a comparison, the Euler--Maruyama scheme is also presented, showing that it does not conserve energy but instead results in an energy increase over time. 
Section~4 focuses on the numerical implementation and evaluates the performance of the proposed method through selected examples. 
Finally, Section~5 illustrates the applicability of our method to spatially non-homogeneous equations using the Vlasov--Poisson--Landau equation as a case study.

\section{The collisional particle system and its properties}\label{section 2}

We first present our particle system in detail, followed by some explanations and remarks. 
Then, we discuss several interesting properties of this system.

\setcounter{algorithm}{0}
\renewcommand{\thealgorithm}{(CP)}
\begin{algorithm}
\floatname{algorithm}{Particle System}
\caption{Collisional particle system for the Landau equation}
\begin{algorithmic}[1]\label{particle-system}
\STATE Choose a time step $\Delta t$, representing the mean duration between two consecutive collisions.
Let $t_m=m\Delta t$, where $m\ge 0$.
\STATE At \(t_0\), independently sample \(\{V_i(0)\}_{i=1}^N\) from the initial distribution \(f_0(v)\).
\STATE At each $t_m$, randomly divide the \(N\) particles into \(N/2\) pairs. 
\STATE During the time interval $[t_m, t_{m+1})$, the particle pair $(i,\theta(i))$ evolve according to:
\begin{equation}\label{particle system of i-th subgroup}
       \md V_i = -\md V_{\theta(i)} = \sigma(V_i-V_{\theta(i)})\circ \md W_i,
\end{equation}
where $W_i=-W_{\theta(i)}$ is a Brownian motion independent of those for other pairs, and
\begin{equation}
\sigma(z)=\sqrt{A(z)}:=\sqrt{\Lambda}|z|^{\gamma/2+1}\Big(I_{d}-\frac{z\otimes z}{|z|^2}\Big).
\end{equation}
The stochastic integral is interpreted in the \emph{Stratonovich} sense. 
\end{algorithmic}
\end{algorithm}

The equation \eqref{particle system of i-th subgroup} is equivalent to \eqref{collisionPS1}, which is written in It\^o's form, provided that $W_i=-W_{\theta(i)}$ in \eqref{collisionPS1} as well. 
To verify this, let $Z_i=V_i-V_{\theta(i)}$, then the dynamics of $Z_i$ is given by
$\md Z_i = 2\sigma(Z_i)\circ\md W_i$.
Using the relationship between Stratonovich and It\^o integrals, one has
\[
\begin{aligned}
\sigma(Z_i) \circ \md W_i
& = \sigma(Z_i) \md W_i + \frac{1}{2} \md [\sigma(Z_i),W_i]\\
& = \sigma(Z_i) \md W_i +  \sum_{j,\ell} \sigma_{\ell j} \partial_{\ell} \sigma_{\!,j}(Z_i) \md t \\
& = \sigma(Z_i) \md W_i +  
\big[\nabla\cdot \sigma^2 - \sigma\nabla \cdot\sigma  \big](Z_i) \md t.
\end{aligned}
\]
A direct calculation shows that 
\[
\nabla\cdot \sigma (z) = (1-d)\sqrt{\Lambda}|z|^{\gamma/2-1}z,
\]
which along with the fact $\sigma(z)z=0$ implies 
$\sigma\nabla \cdot\sigma (Z_i) \equiv 0$.
Recalling that $K = \nabla\cdot A = \nabla\cdot \sigma^2$, we confirm that \eqref{particle system of i-th subgroup} and \eqref{collisionPS1} are indeed equivalent.
We remark that the particle system was first introduced in our previous work \cite{du2024collision}. 
However, to facilitate theoretical analysis, we assumed in \cite{du2024collision} that the Brownian motions \(W_i\) in system \eqref{collisionPS1} are independent, rather than satisfying \(W_i = -W_{\theta(i)}\);
in this case, \eqref{collisionPS1} cannot be rewritten in the form of \eqref{particle system of i-th subgroup}.
This assumption was discussed in Remark 2.3 of \cite{du2024collision}.
We believe that under both settings, the particle system converges to the same limit equation. 
Establishing this convergence under more general conditions will be a focus of our future work.

From the form of equation \eqref{particle system of i-th subgroup}, we observe that \(Z_i = V_i - V_{\theta(i)}\) evolves as a spherical Brownian motion (SBM) within each collision window \([t_m, t_{m+1})\). Specifically, we have
\begin{equation}\label{eq:Z}
\md Z_i = 2\sigma(Z_i)\circ\md W_i
= 2 \sqrt{\Lambda}|Z_i|^{\gamma/2+1} \Big(I_{d} - \frac{Z_i \otimes Z_i}{|Z_i|^2}\Big) \circ \md W_i,
\end{equation}
where \(I_{d} - (Z_i \otimes Z_i)/|Z_i|^2\) is the projection operator onto the plane orthogonal to \(Z_i\). This implies that the magnitude \(|Z_i|\) remains constant during the interval \([t_m, t_{m+1})\), confirming that \(Z_i\) is indeed an SBM within this time window. 
This property is crucial because it not only ensures the \emph{well-posedness of the system} but also allows us to derive an \emph{exact} temporal discretization scheme for Particle System~\ref{particle-system} in the next section.

To intuitively explain why Particle System~\ref{particle-system} approximates the Landau equation, consider the evolution of the velocity distribution \(f^i\) for the \(i\)-th particle during \([t_m, t_{m+1})\), which satisfies:
\[
\partial_t f^i(t, v) = -\nabla_v \cdot \big(K(v - V_{\theta(i)}) f^i(t, v)\big) + \frac{1}{2} \nabla^2_v : \big(A(v - V_{\theta(i)}) f^i(t, v)\big),
\]
Here, \(f^i\) depends on both the random pairing \(\theta(i)\) and the velocity \(V_{\theta(i)}\). 
Assuming that all particles are independently and identically distributed (i.i.d.) at $t_m$ (since there is particle chaos when $N$ is large), the weak correlations during \([t_m, t_{m+1})\) imply that the expected distribution \(\Tilde{f}^i\), averaged over all pairings and velocities, approximately satisfies
\[
\partial_t \Tilde{f}^i \sim -\nabla_v \cdot \big((K \ast \Tilde{f}^i)(v) \Tilde{f}^i\big) + \frac{1}{2} \nabla^2_v : \big((A \ast \Tilde{f}^i)(v) \Tilde{f}^i\big),
\]
which corresponds to the Landau equation \eqref{spatially homogeneous Landau equation}. Thus, it is reasonable to expect that Particle System~\ref{particle-system} converges to the Landau equation.
\medskip

Next, we verify that Particle System~\ref{particle-system} preserves several key physical properties of the Landau equation, namely the conservation of mass, momentum, and energy, as well as entropy dissipation.
The conservation of mass is straightforward. For momentum and energy, we establish the following result:

\begin{proposition}[Pathwise conservation of total momentum and energy]
Under the above setting, the quantities
\[
p(t) := \sum_{i=1}^{N} V_i(t),\quad
\mathcal{E}(t) := \frac{1}{2}\sum_{i=1}^{N} |V_i(t)|^2
\]
remain constant in time.
\end{proposition}

\begin{proof}
It follows from \eqref{particle system of i-th subgroup} that
\(
\md (V_i(t)+V_{\theta(i)}(t)) = 0
\)
for all $t$, thus
\[
\md p(t) = \sum_{i=1}^{N} \md V_i(t) 
= \frac{1}{2} \sum_{i=1}^N \md (V_i(t)+V_{\theta(i)}(t)) = 0,
\]
which implies that $p(t)$ is constant.
Consequently, one has
\[
\begin{aligned}
\md \mathcal{E}(t) & = \frac{1}{4} \sum_{i=1}^{N} \md \big(|V_i(t)|^2 + |V_{\theta(i)}(t)|^2 \big)\\
& =\frac{1}{8} \sum_{i=1}^{N} \big(\md |V_i(t) + V_{\theta(i)}(t)|^2 + \md |V_i(t) - V_{\theta(i)}(t)|^2 \big) \\
& =\frac{1}{8} \sum_{i=1}^{N} \md |V_i(t) - V_{\theta(i)}(t)|^2
= \frac{1}{8} \sum_{i=1}^{N} \md |Z_i(t)|^2.
\end{aligned}
\]
From the equation of $Z_i$ and the chain rule, one has
\[
\md |Z_i|^2 = 2Z_i\circ \md Z_i = 2Z_i^{T}\sigma(Z_i)\circ \md W_i=0,
\]
thus $\md \mathcal{E}(t) = 0$. The proof is complete.
\end{proof}

To justify the dissipation of entropy, we denote by \(f^{N}\) the joint law of the \(N\) particles in Particle System~\ref{particle-system} and define the (normalized) entropy as
\[
H_N(f^N) := \frac{1}{N} \int_{\mathbb{R}^{Nd}} f^{N} \log f^N.
\]
It follows from \eqref{particle system of i-th subgroup} that, during each time interval $[t_m,t_{m+1})$, the law \(f^{N}(t,v_1,\dots,v_N)\) formally satisfies the Liouville equation: 
\begin{equation}\label{Liouville equation 2}
	\partial_t f^{N} = \frac{1}{2} \sum_{i}\nabla_{(v_i,v_{\theta(i)})}\cdot\left(
	\begin{bmatrix}
		A(v_i-v_{\theta(i)}) & -A(v_i-v_{\theta(i)}) \\
		-A(v_i-v_{\theta(i)}) & A(v_i-v_{\theta(i)})
	\end{bmatrix}
	\nabla_{(v_i,v_{\theta(i)})}f^{N}
	\right).
\end{equation}
Assuming that $f^N$ is properly regular, one can deduce that
\begin{align*}\label{computation of dissipation of entropy}
	\begin{aligned}
		\frac{\md}{\md t} H_N(f^N)
		&=\frac{1}{N}\int \partial_t f^{N} \log f^{N}+\frac{1}{N}\int \partial_t f^{N}\\
		&= -\frac{1}{2N}\sum_{i} \int 
		\begin{bmatrix}
			A(v_i-v_{\theta(i)}) & -A(v_i-v_{\theta(i)}) \\
			-A(v_i-v_{\theta(i)}) & A(v_i-v_{\theta(i)})
		\end{bmatrix}
		\nabla_{(v_i,v_{\theta(i)})}f^{N}
		\cdot\frac{\nabla_{(v_i,v_{\theta(i)})}f^{N}}{f^{N}}
		\\ 
		&= -\frac{1}{N}\sum_{i}\int A(v_i-v_{\theta(i)})\bigg(\frac{\nabla_{v_i}f^{N}}{f^{N}}-\frac{\nabla_{v_{\theta(i)}}f^{N}}{f^{N}}\bigg)\cdot\bigg(\frac{\nabla_{v_i}f^{N}}{f^{N}}-\frac{\nabla_{v_{\theta(i)}}f^{N}}{f^{N}}\bigg)f^{N}\\
		& \le 0,
	\end{aligned}
\end{align*}
as $A$ is semi-positive definite. 
Therefore, we have
\begin{proposition}[Entropy dissipation]
	Under the above setting, the quantity \(H_N(f^N)\) decays in time.
\end{proposition}

\begin{remark}
Although the equation \eqref{Liouville equation 2} is a linear \emph{degenerate} parabolic equation, its classical well-posedness can follow from standard PDE theory. This has been discussed by Guillen and Silvestre~\cite[Sec. 3]{guillen2023landau} for the case of 
$N=2$, and the general case is essentially similar. 
Since this paper focuses on the implementation of the algorithm, we defer a detailed analysis to future work.
\end{remark}

\section{Discretization schemes}\label{sec:discretization}

Particle System~\ref{particle-system} is defined in continuous time. In practical implementations, we need to discretize it in time. The simplest approach is to use the Euler–Maruyama (EM) scheme (see Algorithm~\ref{EM}). 
However, as we will show later, the EM scheme does not 
have the desired structure preservation property: 
while the conservation of mass and momentum is automatically satisfied, the EM discretization fails to preserve the total energy of the particle system. 

Fortunately, the unique structure of the system allows us to construct an exact time discretization scheme (see Algorithm \ref{Sphere BM}). Unlike most other methods, our time discretization algorithm introduces no additional errors. 
This ensures that the probability distribution of the particles at the discrete time points matches \emph{exactly} with that of the original continuous-time system, thereby preserving all the system’s desirable physical properties.

The key to achieving this lies in the fact that the relative velocity \(Z_i = V_i - V_{\theta(i)}\) evolves as a spherical Brownian motion. 
According to Stroock’s representation, a standard spherical Brownian motion \(Y_t\) on \(\mathbb{S}^{d-1}\) satisfies the following SDE:
\[
\md Y_t = (I_d - Y_t \otimes Y_t) \circ \md W_t.
\]
Comparing this with \eqref{eq:Z}, it follows that the rescaled spherical Brownian motion \(|Z_i| Y_{kt}\), with time-scaling coefficient \(k = 4\Lambda |Z_i|^{\gamma}\), provides a (weak) solution to \eqref{eq:Z}.


\setcounter{algorithm}{0}
\renewcommand{\thealgorithm}{\arabic{algorithm}}
\begin{algorithm}
	\caption{Spherical Brownian motion (SBM) scheme}
	\label{Sphere BM}
	\begin{algorithmic}[1]
		\REQUIRE Particle number $N$, initial velocity of particles $v_i^{0}$ $(i=1,\cdots, N)$, time step $ \Delta t$, terminal time $T$.
		\FOR{$n=1: \lceil T/\Delta t \rceil$}
		\STATE Randomly divide $N$ particles into $N/2$ pairs.
		\STATE For each pair, compute
		\[
		z = v^{n}_{i} - v^{n}_{\theta(i)}, \quad s = v^{n}_{i} + v^{n}_{\theta(i)}, \quad e_{z} = {z}/{|z|}.
		\]
		\STATE Calculate the time-scaling coefficient 
		$k = 4\Lambda |z|^{\gamma}$.
		\STATE  Simulate a standard SBM on $\mathbb{S}^{d-1}$ starting from $e_z$ with a sampling time $k\Delta t$. 
		Denoted the result as $e'_z$. 
		\STATE Compute the post-collision velocity difference $z'=|z|e'_z$, and update the velocities 
		$$
		v_i^{n+1} = \frac{s+z'}{2}, \quad v_{\theta(i)}^{n+1} = \frac{s-z'}{2}.
		$$
		\ENDFOR
		\STATE \textbf{Return} the velocities at terminal time: \(v_i^{T / \Delta t}\) (\(i = 1, \ldots, N\)).
	\end{algorithmic}
\end{algorithm}

In practice, we use existing methods \cite{Jenkins_2017exactWF,Mijatovi__2020sphereBM} for the exact or approximate simulation of SBM. 
For instance, in two dimensions, SBM on the unit circle can be simulated exactly using the explicit formula \((\cos(B_t), \sin(B_t))\), where \(B_t\) is standard Brownian motion. 
For higher dimensions, SBM increments can still be simulated exactly by decomposing the motion into radial and angular components, as described in \cite{Jenkins_2017exactWF,Mijatovi__2020sphereBM}. 
The angular component is uniformly distributed on \(\mathbb{S}^{d-2}\), while the radial component is governed by the Wright--Fisher diffusion. 
The Wright--Fisher distribution can be either exactly simulated or well-approximated when the time step \(\Delta t\) is small (e.g., \(\Delta t \leq 0.05\)).
\medskip

Next, we present the Euler--Maruyama scheme for the system, and discuss the energy errors produced by this discretization.

\begin{algorithm}[H]
	\caption{Euler--Maruyama (EM) scheme}
	\label{EM}
	\begin{algorithmic}[1]
		\REQUIRE Particle number \(N\), initial velocities \(v_i^{0}\) (\(i=1, \ldots, N\)), time step \(\Delta t\), terminal time \(T\).
		\FOR{\(n =1: \lceil T / \Delta t \rceil\)}
		\STATE Randomly divide \(N\) particles into \(N/2\) pairs.
		\STATE For each pairs, compute \(z = v^{n}_{i} - v^{n}_{\theta(i)}\), sample \(\Delta W_{i} \sim \mathcal{N}(0, \sqrt{\Delta t} I_d)\), 
		calculate:
			\[
			Dv = K(z)\Delta t + \sqrt{A(z)}\Delta W_{i}.
			\]
		and update the velocities:
			$$	v^{n+1}_{i} = v^{n}_{i} + Dv, \quad v^{n+1}_{\theta(i)} = v^{n}_{\theta(i)} - Dv.$$
		\ENDFOR
		\STATE \textbf{Return} the velocities at terminal time: \(v_i^{T / \Delta t}\) (\(i = 1, \ldots, N\)).
	\end{algorithmic}
\end{algorithm}

In the EM scheme, the conservation of momentum is automatically preserved. 
For the energy, we have following result.
\begin{proposition}
    In Algorithm \ref{EM}, the following holds:
    \begin{align*}
|v_i^{n+1}|^2+|v_{\theta(i)}^{n+1}|^2&=|v_i^{n}|^2+|v_{\theta(i)}^{n}|^2+2\Lambda^2(d-1)^2|z_i|^{2\gamma+2}\Delta t^2\\
&\quad\quad + 2\Lambda |z_i|^{\gamma+2}\left(|\Pi(z_i)\xi_i|^2-(d-1)\right)\Delta t
    \end{align*}
    where $\Pi(z):=I_d-z\otimes z/|z|^2$ is the projection matrix, and $\xi_i=\Delta W_i/\sqrt{\Delta t}$ is a standard Gaussian random vector.
    As a consequence, the total energy is increasing.
\end{proposition}

\begin{proof}
	The difference in velocity after one step of the EM discretization is given by
	\[
	z_i + 2K(z_i)\Delta t + 2\sigma(z_i)\Delta W_i.
	\]
	Through direct calculation, and noting that \(\sigma(z_i)\Delta W_i\) is orthogonal to both \(z_i\) and \(K(z_i)\), we have
	\begin{equation}\label{energy estimation}
		\begin{aligned}
			&\quad \left|z_i + 2K(z_i)\Delta t + 2\sigma(z_i)\Delta W_i \right|^2 \\
			&= |z_i|^2 + 4z_i \cdot K(z_i)\Delta t + 4|K(z_i)|^2\Delta t^2 + 4|\sigma(z_i)\Delta W_i|^2 \\
			&= |z_i|^2 + 4\Lambda^2(d-1)^2|z_i|^{2\gamma+2}\Delta t^2 + 4\Lambda |z_i|^{\gamma+2} \big(1 - d + |\Pi(z_i)\xi_i|^2\big)\Delta t.
		\end{aligned}
	\end{equation}
	Using the fact that momentum is conserved, we have
	\begin{align*}
		|v_i^{n+1}|^2+|v_{\theta(i)}^{n+1}|^2&=\frac{1}{2}(|v_i^{n+1}+v_{\theta(i)}^{n+1}|^2+|v_i^{n+1}-v_{\theta(i)}^{n+1}|^2)\\
		&= \frac{1}{2}(|v_i^{n}+v_{\theta(i)}^{n}|^2+|z_i+2K(z_i)\Delta t+2\sigma(z_i)\Delta W_i|^2).
	\end{align*}
	Combining this with the previous calculation in \eqref{energy estimation}, we obtain the desired result.
\end{proof}

\begin{remark}\label{remark}
	By the property of the projection matrix \(\Pi(z)\), we have:
	\[
	\mathbb{E}(|\Pi(z_i)\xi_i|^2 - (d-1)) = 0.
	\]
	Therefore, when summing the energy over all particles, the second term on the right-hand side of \eqref{energy estimation} cancels out according to the law of large numbers if the particle number \(N\) is sufficiently large. As a result, we have
	\[
	\frac{1}{2N} \sum_{i=1}^N |v_i^{n+1}|^2 - \frac{1}{2N} \sum_{i=1}^N |v_i^{n}|^2 \approx \Lambda^2(d-1)^2 \frac{1}{N} \sum_{i=1}^{N/2} |z_i|^{2\gamma+2} \Delta t^2.
	\]
	The energy growth of the EM scheme is also validated in subsequent numerical experiments, especially for the Coulomb potential case ($\gamma=-3$), where the energy will increase rapidly when $|z_i|$ is small.
\end{remark}

\section{Extension to Vlasov--Poisson--Landau equation}\label{sec:non-homo}

In this section, we illustrate the applicability of our method to spatially non-homogeneous equations, in particular the Vlasov--Poisson--Landau equation. Application to other models with Landau collision is similar. 

In spatially non-homogeneous cases, space heterogeneity generates an electric 
field if the particles are charged, which affects both the velocities and 
positions of particles. To describe this phenomenon, the 
Vlasov--Poisson--Landau equation is formulated as:
\begin{equation}
    \begin{dcases}
    \partial_t f + v \cdot \nabla_x f + E \cdot \nabla_v f = Q(f, f), \\
    -\Delta_x \phi = \rho-\rho_0, \\
    E = -\nabla_x \phi,
    \end{dcases}
    \label{Vlasov-Poisson-Landau equation}
\end{equation}
where \(f(x, v, t)\) is the particle distribution, 
$\rho(x,t)=\int_{\Omega_v} f(x,v,t)\md v$,
$\rho_0$ is the background charge density, \(\phi\) is the potential, 
and \(E\) is the electric field. 
When periodic boundary conditions are employed, the charged neutral condition must be imposed for the finiteness of the system energy so that \(\int_{\Omega_x} (\rho-\rho_0) \, dx = 0\). For example, in plasma simulations, $f$ may represent the density of electrons and then $\rho_0= \frac{1}{|\Omega_x|} \int_{\Omega_x}\rho \, dx$ represents the density of ions (considering that electrons have negative charge, one may reverse the sign of $\rho$ and thus $\phi, E$ above, while $-E$ would be used in the Vlasov equation so there would be no intrinsic difference). 

The total energy \(\mathcal{E}_{\text{total}}\) comprises kinetic energy 
\(\mathcal{E}_K\) and electric energy \(\mathcal{E}_E\), given by:
\[
\mathcal{E}_{\text{total}} = \mathcal{E}_K + \mathcal{E}_E, \quad
\mathcal{E}_K = \frac{1}{2} \int_{\Omega_x} \int_{\Omega_v} |v|^2 f \md v \md x, 
\quad \mathcal{E}_E = \frac{1}{2} \int_{\Omega_x} |E|^2 \md x.
\]
The conservation of \(\mathcal{E}_{\text{total}}\) is a critical property of 
the system.

To solve the Vlasov--Poisson--Landau equation, we adopt a time-splitting 
approach that decouples the collision step
\begin{equation}
    \partial_t f = Q(f, f),
    \label{Landau part}
\end{equation}
and the advection step
\begin{equation}
    \begin{dcases}
    \partial_t f + v \cdot \nabla_x f + E \cdot \nabla_v f = 0, \\
    -\Delta_x \phi = \rho-\rho_0, \\
    E = -\nabla_x \phi.
    \end{dcases}
    \label{Vlasov-Poisson part}
\end{equation}

Traditionally, some studies have employed the DSMC method to address the collision step in the Vlasov--Poisson--Landau equation \cite{lei2024adaptivesamplingaccelerateshybrid,YAN201618}. 
This approach involves simulating Boltzmann collisions to approximate the Landau collision. 
While it has been effectively combined with other schemes for solving the advection step, the method requires rejection sampling when updating velocities, which may increase computational cost and time.

For ~\eqref{Landau part}, we use the SBM scheme to update particle velocities in the same grid, 
preserving the kinetic energy \(\mathcal{E}_K\). 
Since particle positions 
remain unchanged during this step, the electric field and \(\mathcal{E}_E\) 
are also unchanged.

For ~\eqref{Vlasov-Poisson part}, which corresponds to the Vlasov--Poisson equation, the particle-in-cell (PIC) algorithm is widely employed to solve such equations \cite{bailo2024collisional, serikov1999particle}. 
As a preparatory step for the algorithm, the spatial domain \(\Omega_x\) is divided into a set of grids \(\Omega_k\). 
Using the electric field \(E_k\) at the center of each grid \(\Omega_k\), the electric energy \(\mathcal{E}_E\) can be approximately computed as:
\[
\mathcal{E}_E = \frac{1}{2} \sum\limits_k \|E_k\|_2^2|\Omega_k|,
\]
where increasing the number of grids improves the accuracy of \(\mathcal{E}_E\).
Let \(x_k\) denote the center of the \(k\)-th grid, \(x_i\) the position of the \(i\)-th particle, \(Q\) the total charge in the whole area, and \(q\) the charge of each particle.
The PIC algorithm proceeds in two main steps:

\emph{Step 1: Solving the Poisson equation}.
The charge density at the grid point \(x_k\) is approximated by:
\begin{align}
\label{charge density int}
    \rho(x_k)= Q\int_{\Omega_v}  f(x_k,v,t)\md v  - \rho_0\approx q\sum\limits_iS(x_k-x_i)- \rho_0=:\rho_k,
\end{align}
where  \(q=Q/N\), \(S(x)\) 
is a shape function used to approximate the \(\delta\)-function, and \(\rho_0 := {|\Omega_x|^{-1}} \int_{\Omega_x} \rho \md x\).
Then the Poisson equation \(-\Delta_x \phi = \rho-\rho_0\) 
is solved with periodic boundary conditions. 
The spectral scheme is utilized 
to obtain \(\phi\), from which the electric field at the center of the \(k\)-th grid, \(E(x_k)\), is computed as \(E(x_k) = -\nabla_x \phi(x_k)\).

\emph{Step 2: Updating velocities and positions}.
To update the particle velocities and positions, the electric field at the particle positions \(E(x_i)\) is needed. 
Using the same shape function \(S(x)\), the electric field at a general location \(x\) is interpolated as:
\begin{equation}
    E(x) =\int_{\Omega_x} E(y) S(y - x) \md y\approx \sum\limits_{k} E(x_k) S(x_k - x)|\Omega_k|.
    \label{electric field int}
\end{equation}
Subsequently, the leap-frog scheme is typically employed to update the particle velocities and positions, offering high accuracy \cite{birdsall1991particle}.

However, the leap-frog scheme in this step does not conserve the total energy \(\mathcal{E}_{\text{total}}\). 
To address this, the Vlasov--Poisson equation can be equivalently reformulated as the Vlasov--Amp\`ere equation, and the Crank--Nicolson (CN) method \cite{CHEN20117018} can be used instead. 
The CN method ensures strict conservation of both the total energy \(\mathcal{E}_{\text{total}}\) and the total charge. The specific form is formulated as:
\begin{align}
    \begin{dcases}
        \frac{E^{n+1}(x_k)-E^n(x_k)}{\Delta t}+J^{n+1/2}(x_k)=J_{\text{mean}},\\
J^{n+1/2}(x_k)=q\sum\limits_iS(x_k-x_i^{n+1/2})v_i^{n+1/2},\\
x_i^{n+1}-x_i^n=v^{n+1/2}_i\Delta t,\\
v_i^{n+1}-v_i^n=\frac{1}{2}[E^n(x_i^{n+1/2})+E^{n+1}(x_i^{n+1/2})]\Delta t,
    \label{V--A iteration}
    \end{dcases}
\end{align}
where $x^{n+1/2}_i:=\frac{1}{2}(x^{n}_i+x^{n+1}_i), \quad v^{n+1/2}_i:=\frac{1}{2}(v^{n}_i+x^{n+1}_i), \quad J_{\text{mean}}:=\frac{1}{n_0}\sum_kJ^{n+1/2}(x_k).$

By combining the SBM scheme, which preserves the kinetic energy \(\mathcal{E}_K\) in the Landau collision step, with the energy-conserving PIC algorithm, the total energy \(\mathcal{E}_{\text{total}}\) can be rigorously preserved throughout the simulation. See Algorithm \ref{PIC+Landau collision} for details.

\begin{algorithm}[!h]
		\caption{PIC + Landau Collision}
		\label{PIC+Landau collision}
  \begin{algorithmic}[1]
      \REQUIRE Particle number \(N\), initial data \((x^{(0)}_i, v^{0}_i)_{i=1}^N\), particle charge \(q\), time step \(\Delta t\), 
      terminal time \(T\), domain length \(2L\), 
      number of cells \(n_0\), grid size \(\Delta x=2L/n_0\).
      \STATE Compute the initial electric field \(E_k \ (k=1,\cdots,n_0)\) (for example using Algorithm~\ref{PIC solve electric field}).
      \FOR{\(m = 1 : \lceil T / \Delta t \rceil\)}
          \STATE For each cell \(G_k \ (k=1,\cdots,n_0)\), perform the Landau collision using the SBM scheme. \\
          If \(G_k\) contains an odd number of particles, with probability $1/2$, the extra one collides with a randomly selected one in the post-collision particles.
          \STATE Solve the implicit scheme \eqref{V--A iteration} using some iterative method.
      \ENDFOR
      \STATE \textbf{Return} final particle velocities \(v_i^{T/\Delta t}\) and positions \(x_i^{T/\Delta t} \ (i=1,\cdots,N)\).
  \end{algorithmic}		
\end{algorithm}

\section{Numerical experiments}

In this section, we validate the effectiveness of our method through numerical experiments for solving the Landau equation with different values of \(\gamma\) and dimensions. We compare our numerical solutions with the analytical Bobylev--Krook--Wu (BKW) solution for Maxwell molecules and the reference solution given in \cite{carrillo2021random} for the Coulomb potential case.
Besides, we demonstrate the reliability and accuracy of our approach for the Vlasov-Poisson-Landau equation by studying the effects of Landau collision in the phenomenon of Landau damping. 

To visualize the particle solution and compare it with the exact (or reference) solution, we construct a mollified solution \(f^N_{\epsilon}\) from the empirical measure of particle velocities as follows:
\[
f^N_{\epsilon}(v) := \psi_{\epsilon} \ast \bigg( \frac{1}{N} \sum_{i=1}^N \delta_{v_i} \bigg) = \frac{1}{N} \sum_{i=1}^N \psi_\epsilon(v - v_i),
\]
where \(\psi_{\epsilon}(x)\) is the Gaussian mollifier. 
Compared to deterministic particle methods such as \cite{carrillo2020particle, carrillo2021random}, the parameter \(\epsilon\) is only used as a post-processing parameter and can therefore be chosen more flexibly. 
For visualization purposes and consistency, we set \(\epsilon = 0.01\) uniformly across all experiments.

To measure the accuracy of our solution, we use the relative \(L_2\) error, which is computed on a uniform mesh grid with center points \(v_l^c\) of the squares as follows:
\[
\mathcal{E}_2 := \frac{\Vert f^{\text{ext}} - f_\epsilon^N \Vert_2}{\Vert f^{\text{ext}} \Vert_2} \approx \frac{\sqrt{\sum_{l=1}^{N_{\text{grid}}} h^d \left|f^{\text{ext}}(v_l^c) - f_\epsilon^N(v_l^c)\right|^2}}{\sqrt{\sum_{l=1}^{N_{\text{grid}}} h^d \left|f^{\text{ext}}(v_l^c)\right|^2}}.
\]
We also evaluate the entropy of the mollified solution \(f^N_{\epsilon}\) using the following expression:
\[
H(f^N_{\epsilon}) = \int f^N_\epsilon \log f^N_\epsilon \md v \approx \sum_{l=1}^{N_{\text{grid}}} h^d f^N_{\epsilon}(v_l^c) \log(f^N_\epsilon(v_l^c)).
\]
Recall that we have shown in Section~\ref{section 2} that the entropy of the joint distribution \(H_N(f^N)\) is monotonically decreasing. 
However, in practical computations, directly evaluating the entropy of the joint distribution of particles is infeasible because obtaining the corresponding density function is challenging. 
By smoothing the empirical distribution of the particles, the resulting mollified density function \(f^N_\epsilon\) approximates the marginal distribution of the joint distribution. 
Under the assumption that all particles are approximately i.i.d., the following relationship holds approximately:
\[
H_N(f^N) \approx H_N(f^{\otimes N}) = H(f) \approx H(f^N_\epsilon).
\]
Thus, in our numerical experiments, we compute the entropy of the mollified empirical measure \(f_\epsilon^N\) to verify the entropy dissipation property.

\subsection{2D BKW solution for Maxwell molecules}
Let $d=2$ and $\gamma=0$.
In this case, the collision kernel is 
\[A(z)=\frac{1}{8}\big(|z|^{2}I_d-z\otimes z\big),\] 
and the BKW solution is given by
$$ f(t,v)=\dfrac{1}{2\pi K}\Big(2-\dfrac{1}{K}+\dfrac{1-K}{2K^{2}}|v|^{2}\Big)\exp\Big(-\dfrac{|v|^{2}}{2K}\Big),
\quad K=1-\frac{1}{2}\exp\Big(-\frac{t}{8}\Big) .$$ 

We set \(t_0 = 0\), \(t_\text{end} = 200\), and \(\Delta t = 0.1\). 
The initial particle velocities are sampled independently from the initial distribution. 
We first use the SBM and EM schemes to solve the equation with both \(N = 10,000\) and \(N = 100,000\). 
The results are presented in Figure~\ref{Comparison of different $N$ using SBM method 2d Maxwell case}. 

Figure~\ref{Relative L2 error of different N 2d Maxwell case} shows the time evolution of the relative \(L_2\) error for the SBM and EM schemes (Algorithms~\ref{Sphere BM} and \ref{EM}, respectively) with different \(N\).
In the initial period, the relative \(L_2\) errors of the two methods are similar. 
However, after running for a longer time, the error of the EM scheme accumulates and becomes more noticeable. 
Figure~\ref{energy of different N 2d Maxwell case} shows that the energy of the EM scheme grows approximately linearly over time, consistent with our theoretical results, whereas the SBM scheme preserves energy.
Furthermore, Figures~\ref{energy of different N 2d Maxwell case} and \ref{entropy of different N 2d Maxwell case mother and son} confirms that the SBM scheme satisfies both energy conservation and entropy dissipation properties.
 \begin{figure}[!t]
        \centering
        \subfigure[Relative $L_2$ error]{
        \centering
        \includegraphics[height=50mm,width=49mm]{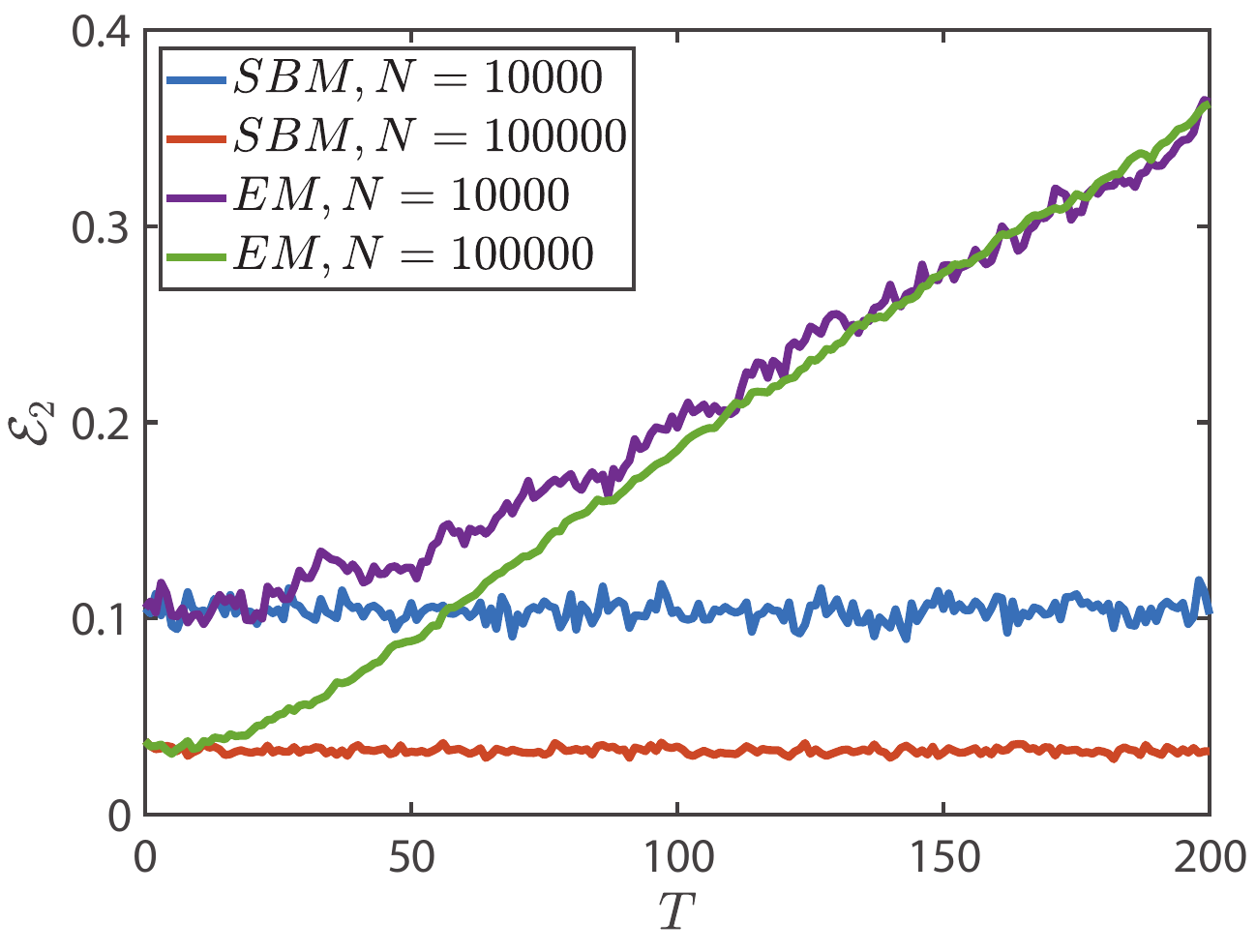}
        \label{Relative L2 error of different N 2d Maxwell case}
    }
    \centering
         \subfigure[Energy]{
        \centering
        \includegraphics[height=50mm,width=49mm]{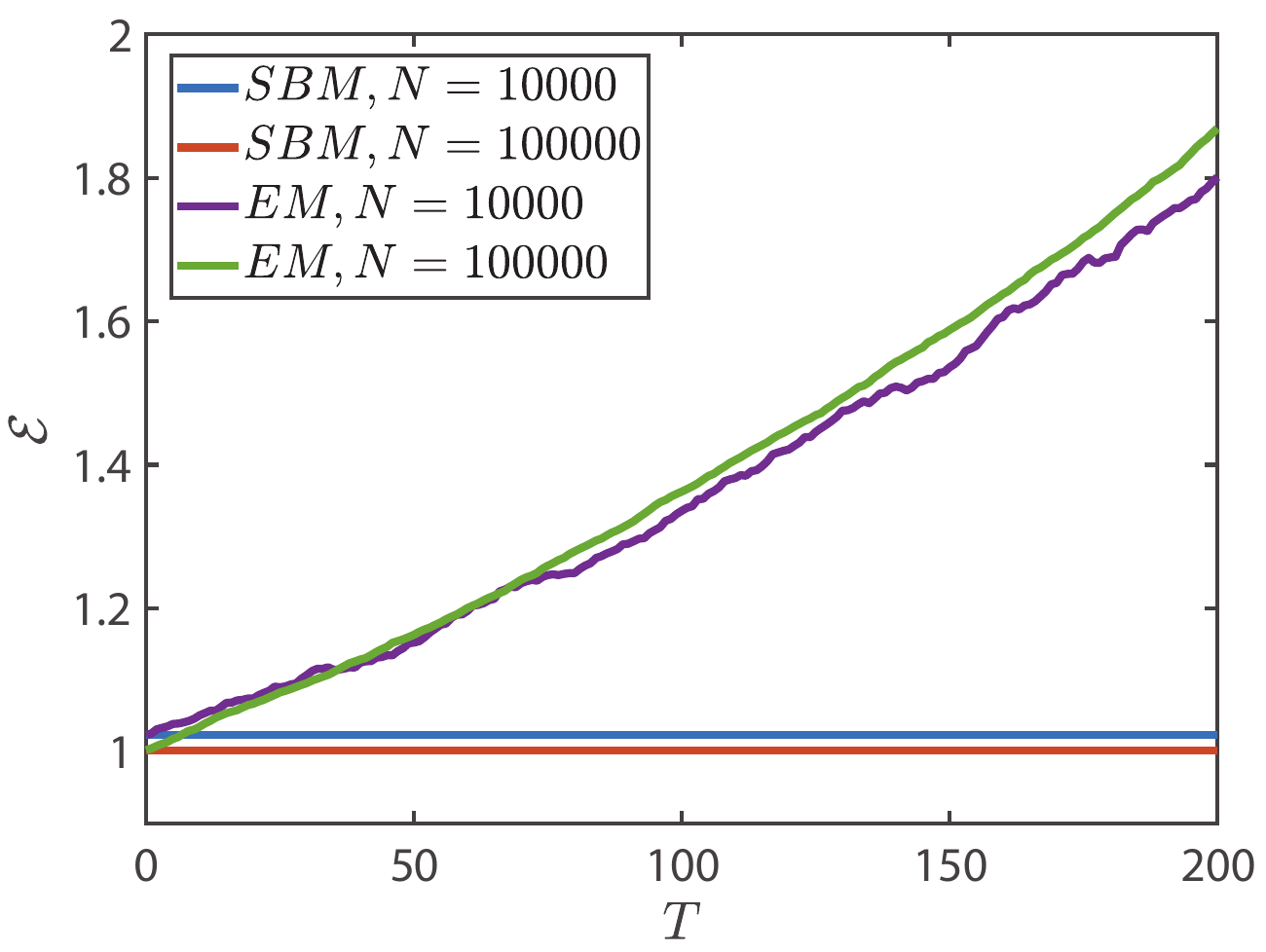}
        \label{energy of different N 2d Maxwell case}
    }
	\centering
	\subfigure[Entropy]{
        \centering
        \includegraphics[height=50mm,width=50mm]{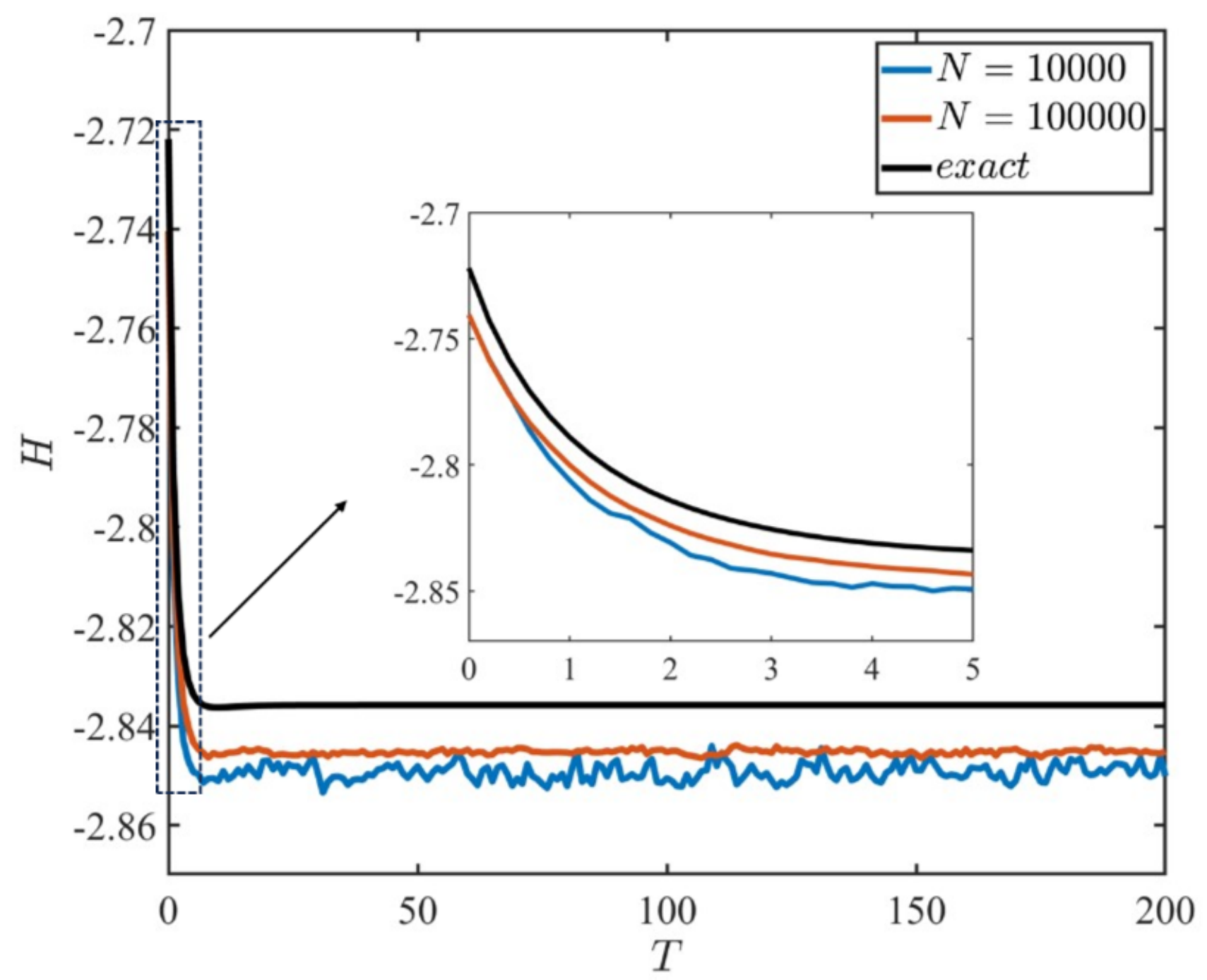}
        \label{entropy of different N 2d Maxwell case mother and son}
    }
		\centering
		\caption{Time evolution of relative $L_2$ error, energy and entropy for different $N$, where (a)(b) shows the results of both SBM and EM schemes and (c) show the results of SBM scheme.}
		\label{Comparison of different $N$ using SBM method 2d Maxwell case}
    \end{figure}

Next, we evaluate the order of accuracy and CPU time per time step for the SBM scheme. 
Figure~\ref{order of accuracy 2dmax} shows the relative \(L_2\) error at \(t = 5\) for different \(N\) and \(\Delta t\), indicating that the errors scale approximately as \(N^{-1/2}\). 
Additionally, Figure~\ref{CPU time 2DBKW} illustrates that the CPU time per time step for the SBM scheme increases linearly with \(N\).
    \begin{figure}[!t]
        \subfigure[Order of accuracy at $t = 5$]{
        \centering
        \includegraphics[height=50mm]{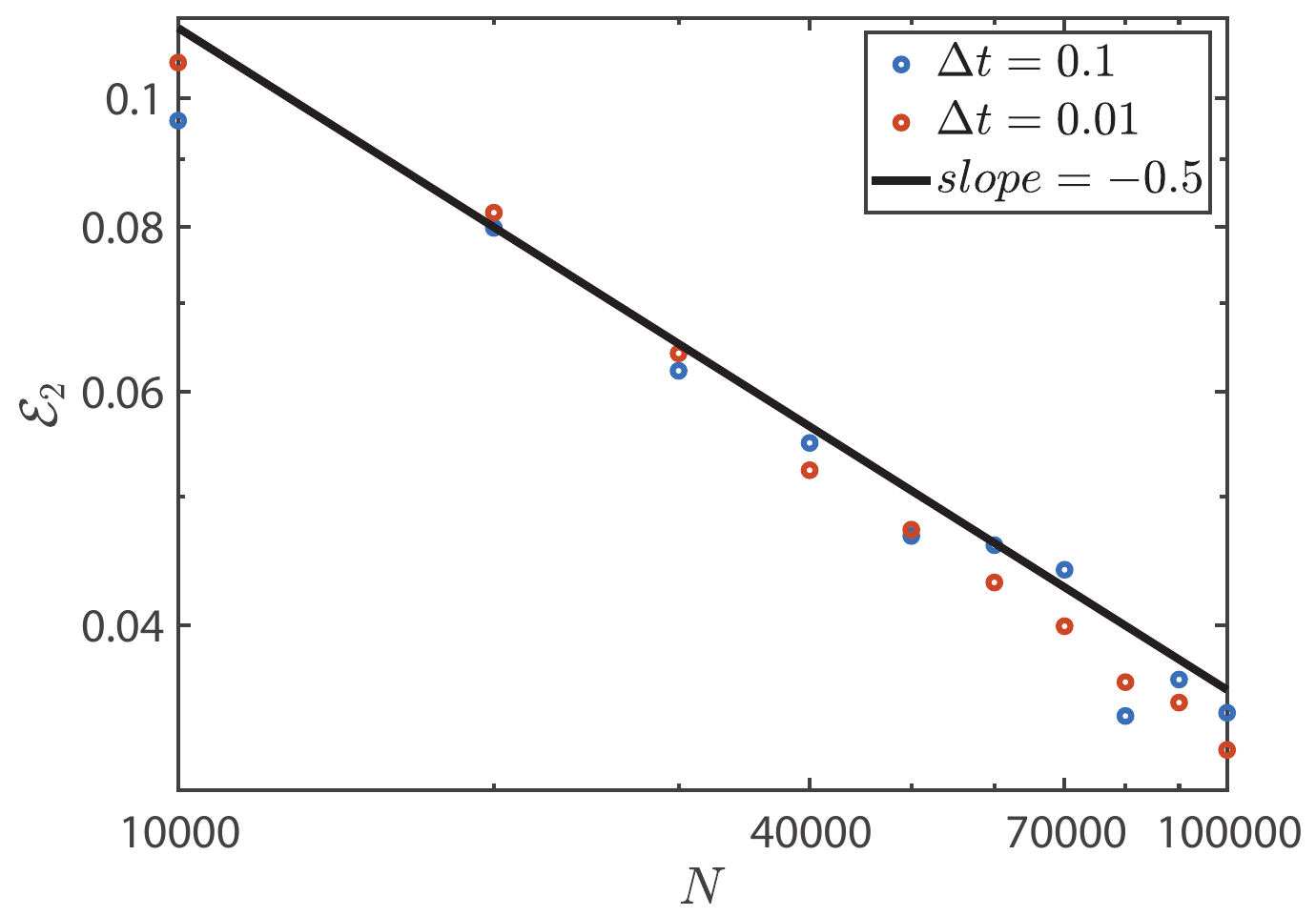}
        \label{order of accuracy 2dmax}
    }
           \subfigure[CPU time per time step (in seconds) with respect to particle number $N$]{
         \centering
           \includegraphics[height=50mm]{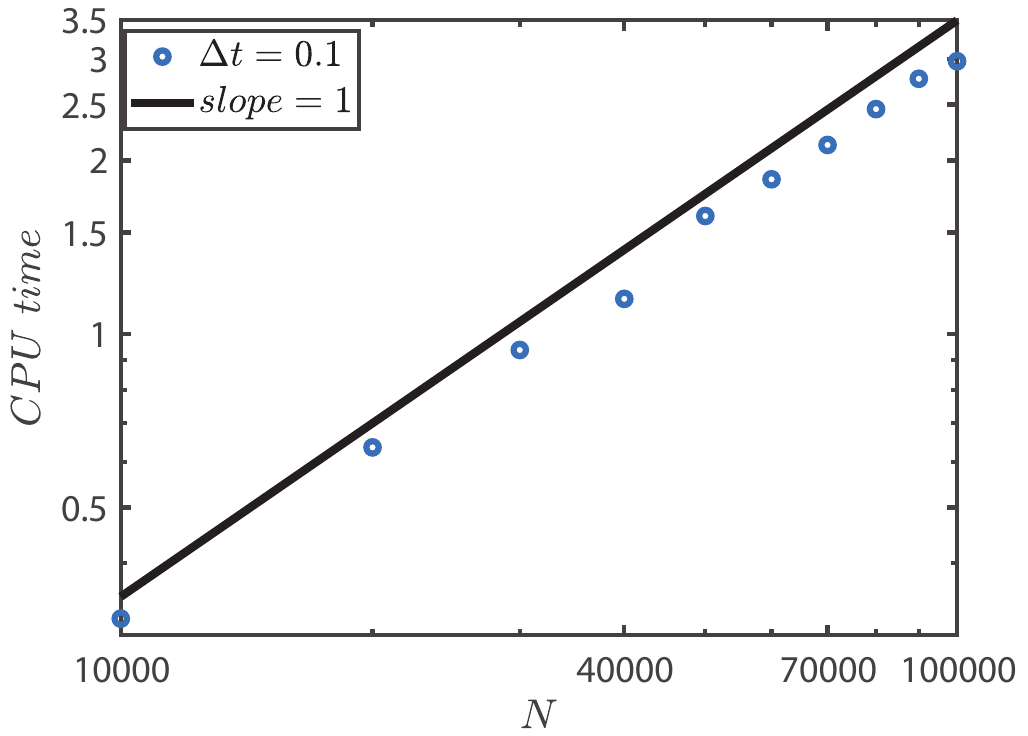}
           \label{CPU time 2DBKW}
           }
		\centering
		\caption{Convergence rate (left) and CPU time (right) of SBM scheme}
		\label{Convergence results (left) and CPU time (right) for 2D maxwell}
    \end{figure}

This example demonstrates that our particle system effectively approximates the spatially homogeneous Landau equation and achieves half-order accuracy as the number of particles increases. 
Notably, the SBM scheme successfully preserves both conservation properties and entropy dissipation, making it suitable for stable long-time simulations. 
The computational cost of the method is confirmed to be \(O(N)\) based on the CPU time tests.

\subsection{2D anisotropic solution with a singular kernel}
    For $\gamma=-3$ and $d=2$, the collision kernel is
    $$A(z)=\frac{1}{8|z|^{3}}\big(|z|^{2}I_d-z\otimes z\big),$$
    and the initial condition is chosen as
    $$f(0,v)=\frac{1}{4\pi}\Big[0.4\exp\Big(-\frac{(v-u_{1})^{2}}{2}\Big)+1.6\exp\Big(-\frac{(v-u_{2})^{2}}{2}\Big)\Big] ,\quad u_{1}=(-2,1), u_{2}=(1,-1).$$
    For this problem, we do not have analytical solution. 
    Here, we use the Type 1 Random Batch Method (Algorithm 3 in \cite{carrillo2021random}) with $ \Delta t=0.2,  n_{0}=200$ (particle number per dimension), $\epsilon=0.04$ (parameter for the Gaussian mollifier) as the reference solution. (The computational cost for the reference solution with larger $n_0$ would be high.)

    We set \(t_0 = 0\), \(t_\text{end} = 200\), \(\Delta t = 0.1\), and use the SBM and EM schemes to solve the equation for \(N = 10,000\) and \(N = 100,000\). 
    The time evolution results are shown in Figure~\ref{Comparison of different $N$ using SBM method 2d Coulomb case}.
    
We observe that the SBM scheme continues to perform well over long time periods, preserving the energy conservation. 
In contrast, as shown in Figure~\ref{Relative L2 error of different N 2d Coulomb case}, the EM scheme performs poorly in the Coulomb case, even over short time intervals. 
Since the collision kernel \(A\) is singular, when the velocities of a pair of collision particles are close to each other, i.e. when \(|z|\) is small, the EM scheme will suffer from significant errors and as discussed in Remark~\ref{remark}, this can lead to energy blow-up, which is demonstrated in 
Figure~\ref{energy of different N 2d Coulomb case}. 
From Figure~\ref{Relative L2 error of different N 2d Coulomb case}~\ref{energy of different N 2d Coulomb case} we observe that the SBM scheme maintains energy conservation and long-term stability of the system even in the singular kernel case.
 
    We also test the convergence order and the CPU time per time step for the SBM scheme. 
    The results, shown in Figure~\ref{convergence order and CPU time of 2d coulomb}, confirm that our method achieves half-order accuracy and \(O(N)\) computational cost per time step.

    From this example, we see that the SBM scheme can solve the Landau equation stably and preserve the structure even for the singular kernel case. In contrast, the EM scheme cannot do well even in short time if it is not equipped with a cutoff for small \(|z|\).
   
    \begin{figure}[!t]
        \centering
        \subfigure[Relative $L_2$ error]{
        \centering
        \includegraphics[width=0.42\textwidth, height=0.41\textwidth]{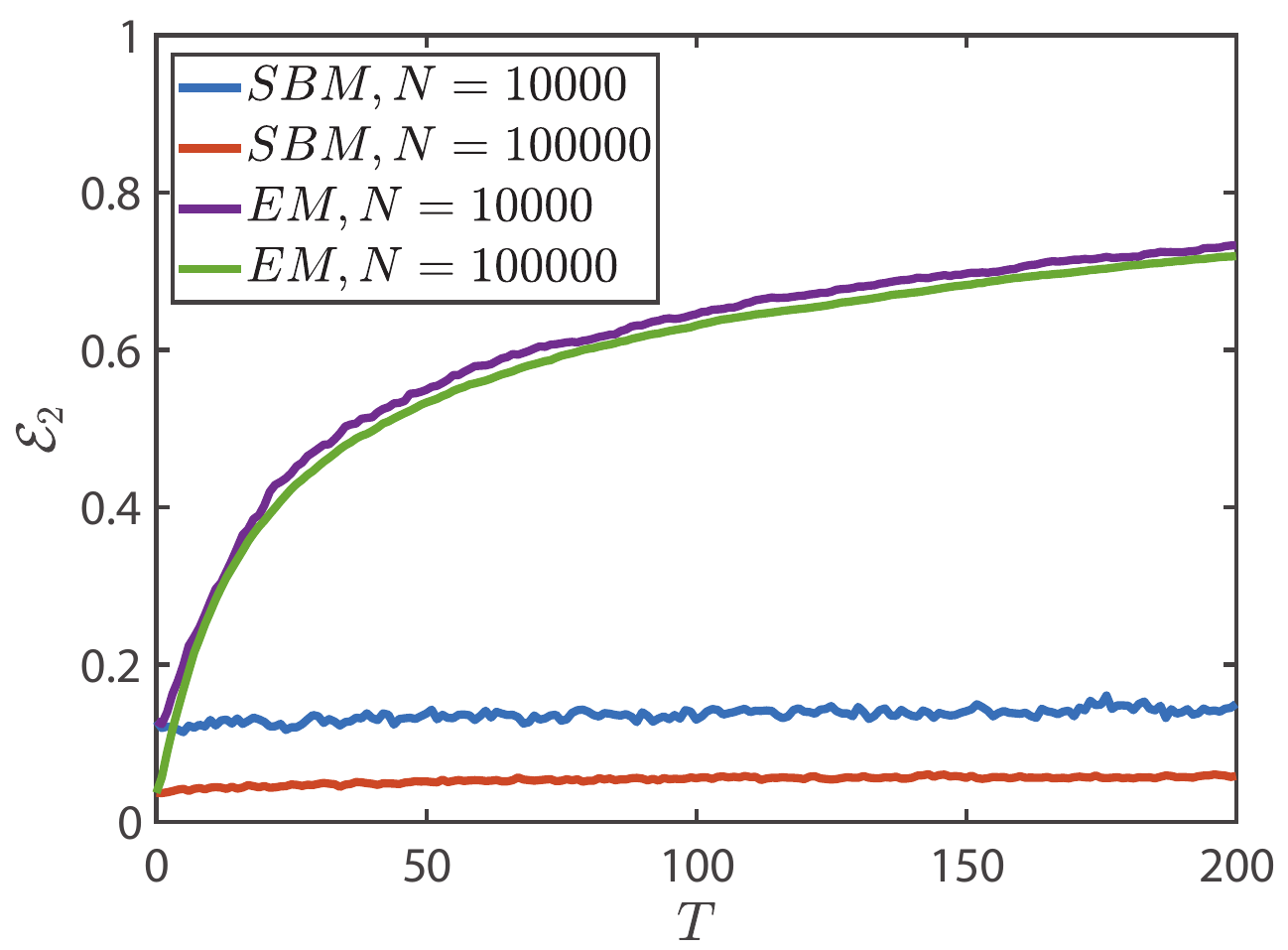}
        \label{Relative L2 error of different N 2d Coulomb case}
    }
           \centering
        \subfigure[Energy (the energy of SBM is a constant)]{
        \centering
        \includegraphics[width=0.42\textwidth,height=0.42\textwidth]{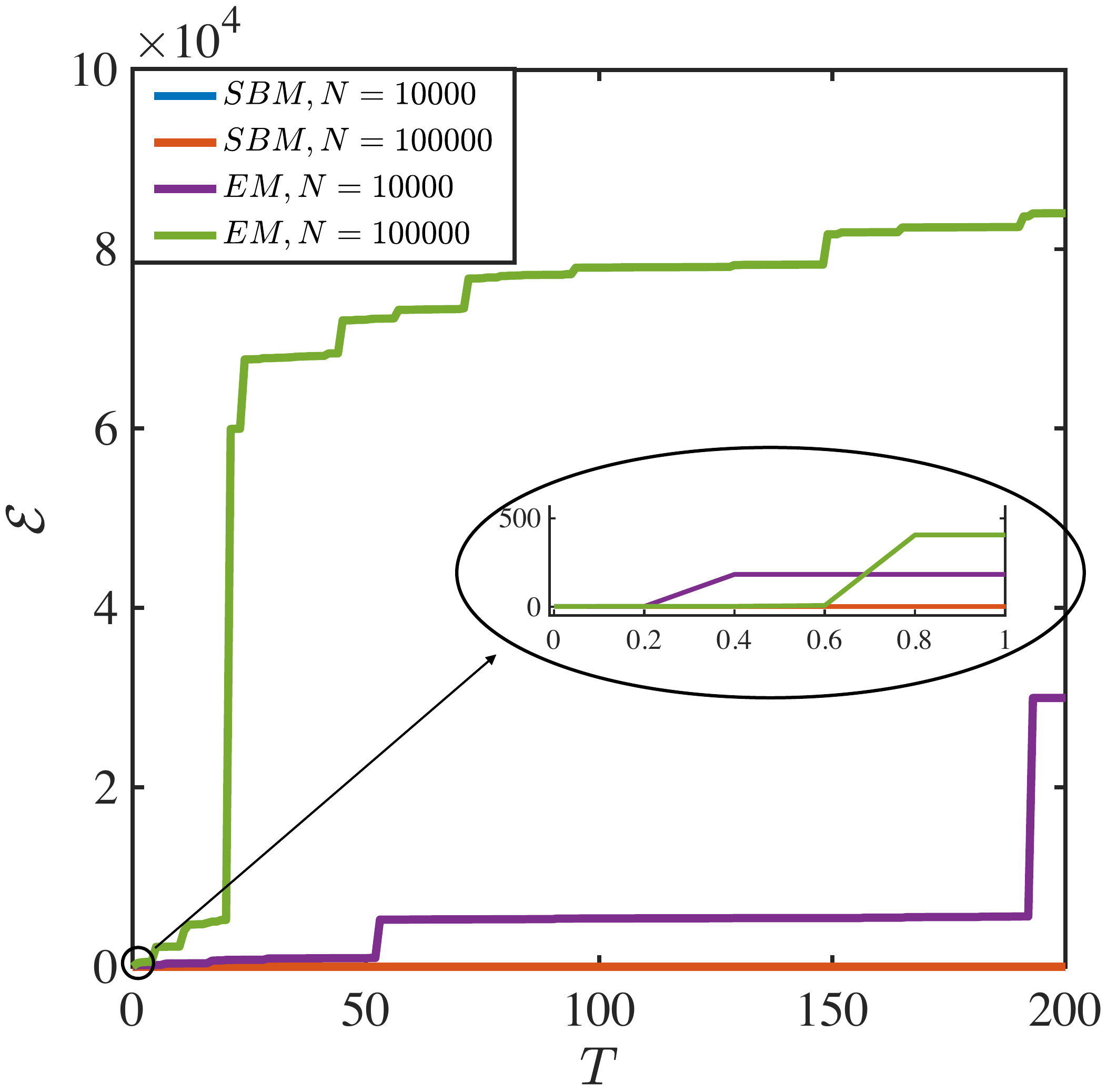}
        \label{energy of different N 2d Coulomb case}
    }
		\centering
		\caption{Time evolution of relative $L_2$ error, energy for different $N$ , where (a)(b) shows the results of both SBM and EM schemes.}
		\label{Comparison of different $N$ using SBM method 2d Coulomb case}
    \end{figure}

    \begin{figure}[!t]
        \subfigure[Order of accuracy for 2D Coulomb at $t = 5$]{
        \centering
        \includegraphics[height=50mm]{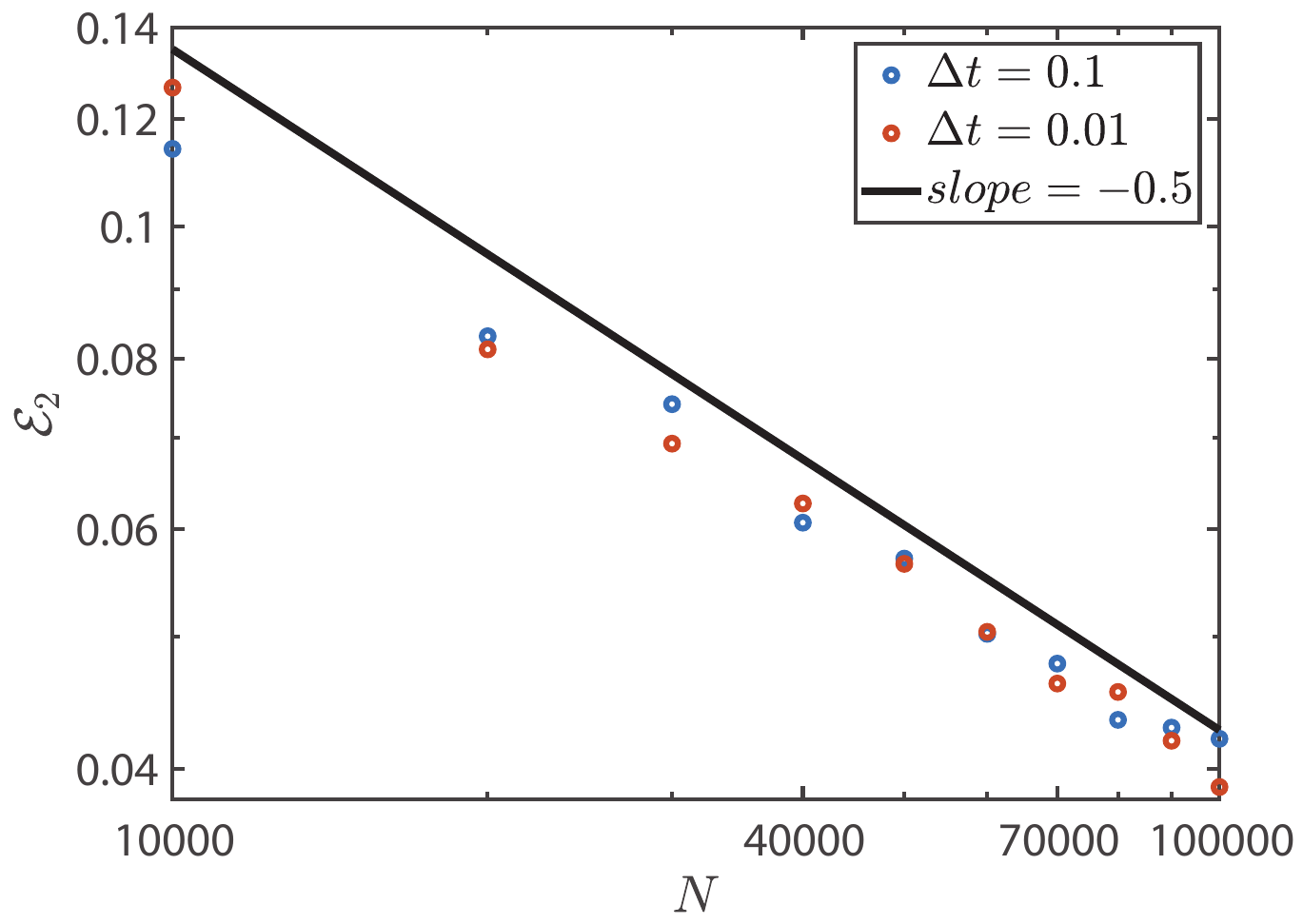}
        \label{relative L2 error at T=5 versus N 2d Coulomb case}
    }
		\subfigure[CPU time (in seconds) per time step with respect to particle
number $N$]{
        \centering
        \includegraphics[height=50mm]{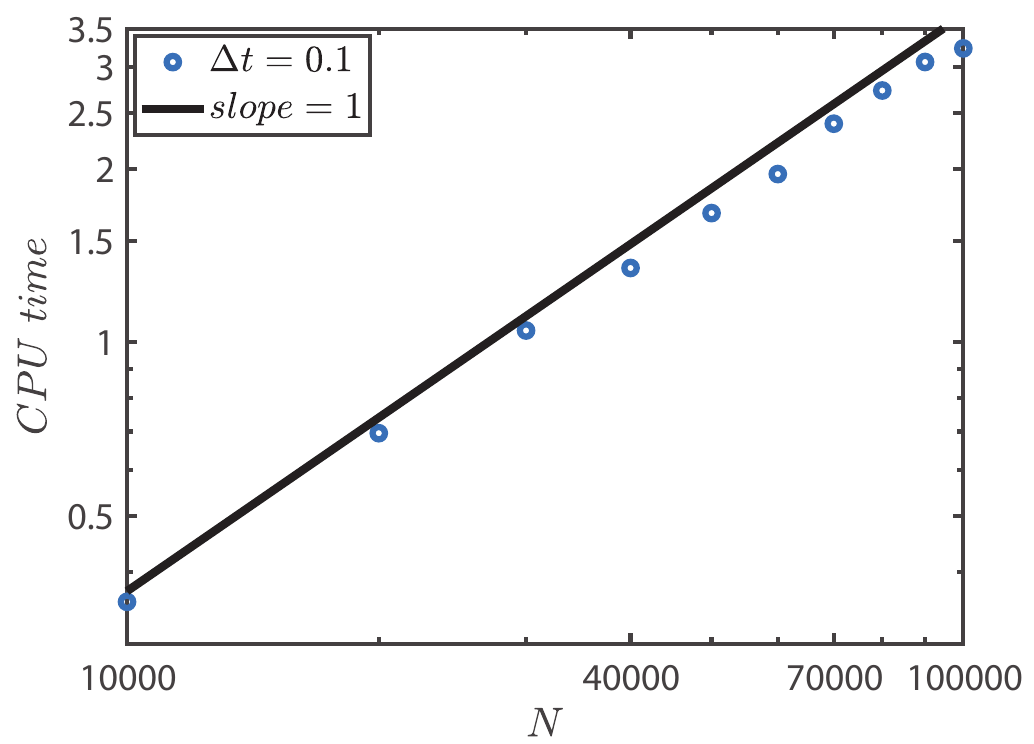}
        \label{CPU2dcoulomb}
    }
		\centering
		\caption{ Convergence order (left) and CPU time (right) of SBM scheme}
		\label{convergence order and CPU time of 2d coulomb}
    \end{figure}

\subsection{3D BKW solution for Maxwell molecules}
For $d=3$ and $\gamma=0$, the collision kernel is
$$ A(z)=\frac{1}{12}\big(|z|^{2}I_d-z\otimes z\big),$$
and the BKW solution is given by
$$ f(t,v)=\frac{1}{(2\pi K)^{1.5}}\Big(2.5-\frac{3}{2K}+\frac{1-K}{2K^{2}}|v|^{2}\Big)\exp\Big(-\frac{|v|^{2}}{2K}\Big),\quad
K=1-\exp\Big(-\frac{t}{6}\Big).$$ 

Similar to the \(2D\) case, we first compare the numerical results obtained using the SBM and EM schemes. 
With \(t_0 = -6\ln(0.4)\) (such that $K=0.6$), \(t_{\text{end}} = t_0+200\), \(\Delta t = 0.1\), \(N = 50,000\) and \(500,000\), the comparison results are shown in Figure~\ref{Comparison of different $N$ using SBM method 3d Maxwell case}.
From Figure~\ref{Relative L2 error of different N 3d Maxwell case}, the relative \(L_2\) errors for the SBM scheme remain generally stable during $[t_0,t_\text{end}]$, and the errors for the EM scheme initially resemble that of the SBM, but it gradually increases over time.
It can also be observed that as the number of particles increases, the relative error for both SBM and EM scheme decreases.
As shown in Figure~\ref{energy of different N 3d Maxwell case}, the EM scheme does not preserve energy, whereas the SBM scheme successfully maintains energy conservation.


 \begin{figure}[!t]
        \centering
        \subfigure[Relative $L_2$ error]{
        \centering
        \includegraphics[scale=0.32]{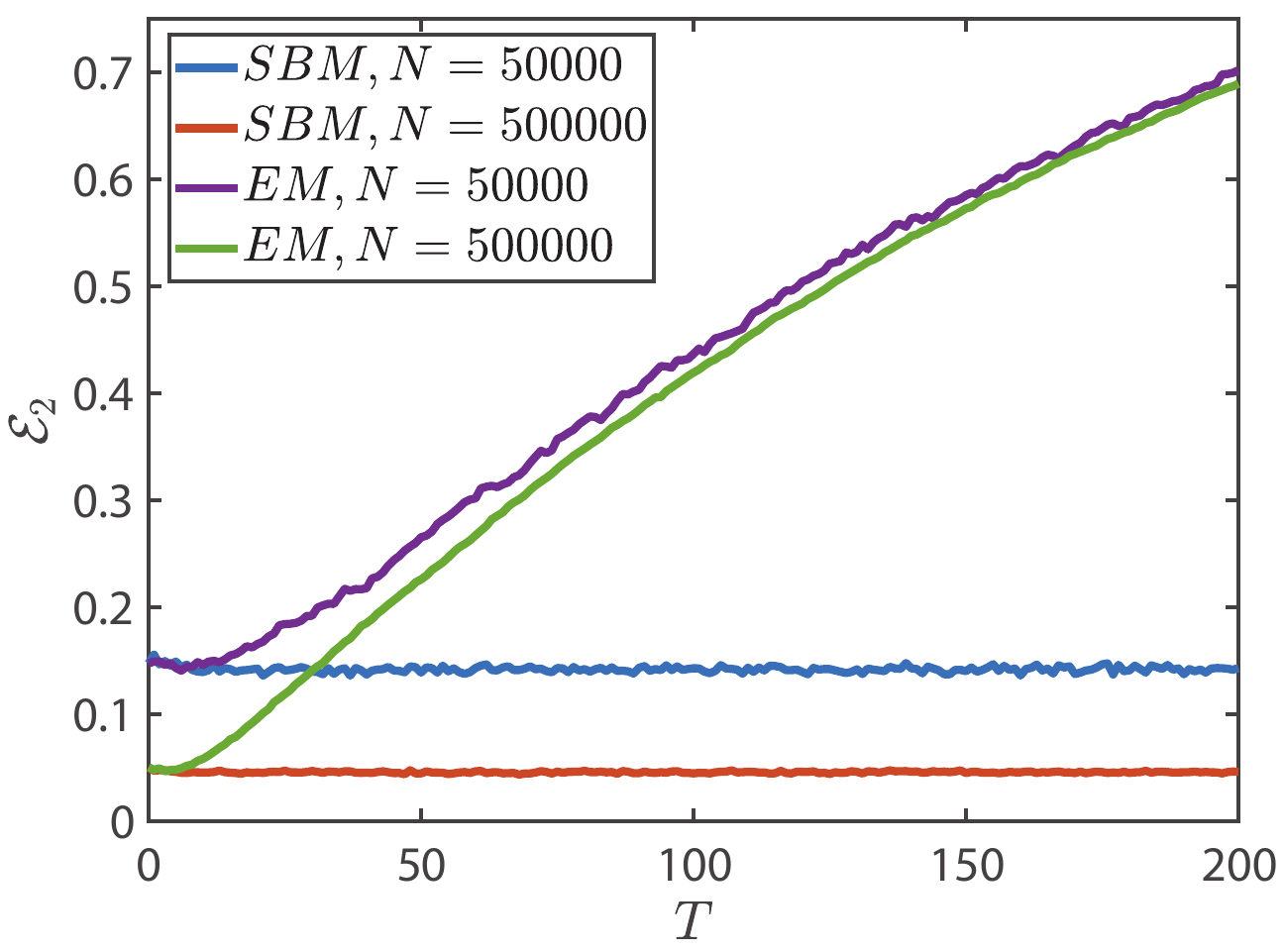}
        \label{Relative L2 error of different N 3d Maxwell case}
    }
    \centering
         \subfigure[Energy]{
        \centering
        \includegraphics[scale=0.32]{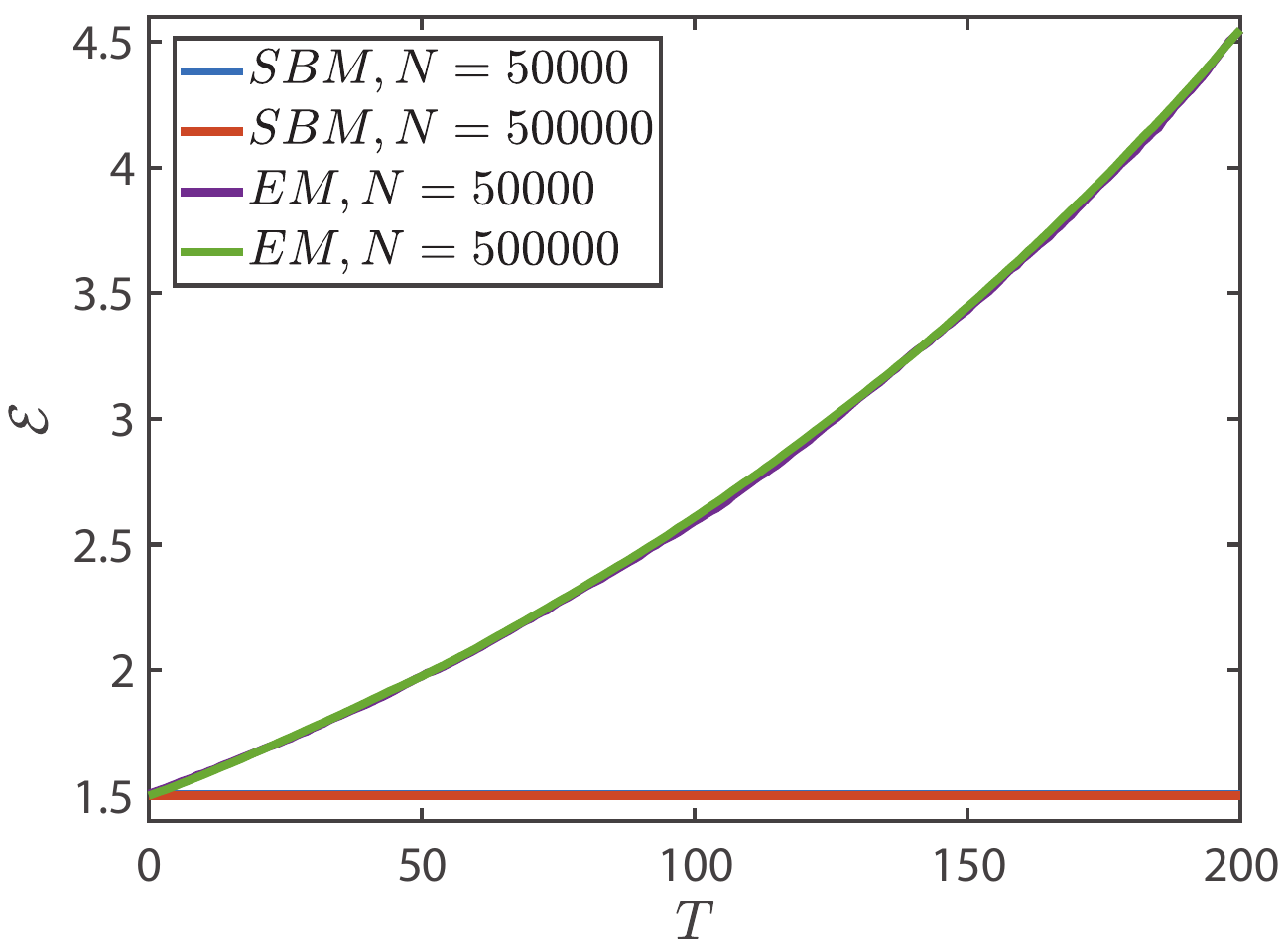}
        \label{energy of different N 3d Maxwell case}
    }
        \centering
		\caption{Time evolution of relative $L_2$ error and energy for different $N$.}
		\label{Comparison of different $N$ using SBM method 3d Maxwell case}
    \end{figure}

This example demonstrates that our method can effectively solve higher-dimensional Landau equations, especially we can achieve long-term stability and energy conservation if we use SBM scheme.
Furthermore, the approximation improves as the number of particles increases.

\subsection{A Vlasov--Poisson--Landau example with Coulomb potential}

Next, we demonstrate the effectiveness of our method when extended to spatially nonhomogeneous cases. In particular, we validate the phenomenon of Landau damping, demonstrating the reliability and accuracy of our approach.

We will use Algorithm \ref{PIC+Landau collision} to solve the Vlasov--Poisson--Landau equation. The electric field is initialized by Algorithm \ref{PIC solve electric field} in this example while the implicit scheme \eqref{V--A iteration} will be solved approximately by Algorithm \ref{alg:CNimplicitsolver}.

\begin{algorithm}[!h]
    \caption{Solve electric field in grid centers}
    \label{PIC solve electric field}
    \begin{algorithmic}[1]
        \REQUIRE 
        Number of particles $N$, particle positions $x_i \, (i=1,\cdots,N)$, particle charge $q$, number of spatial grids $n_0$, and grid size $\Delta x$.        
        \STATE Compute the charge density at the grid centers:
        \[
        \rho(x_k) = \sum_{i=1}^N \hat{S}(x_i - x_k) \frac{q}{\Delta x}, 
        \quad \rho_{\text{ion}} = \frac{1}{n_0} \sum_{k=1}^{n_0} \rho(x_k).
        \]
        \STATE Solve the Poisson equation using a spectral scheme with periodic boundary conditions:
        \[
        -\Delta \phi = \rho - \rho_{\text{ion}}.
        \]
        \STATE Compute the electric field at the grid centers:
        \[
        E_k = -\nabla_x \phi(x_k), \quad k = 1, \cdots, n_0.
        \]
    \end{algorithmic}
\end{algorithm}

\begin{algorithm}[!h]
		\caption{A possible iterative method for solving \eqref{V--A iteration}}
        \label{alg:CNimplicitsolver}
  \begin{algorithmic}[1]
          \REQUIRE Iteration number $n$, electric field $E^m(x_k)$ on grids
         \STATE Compute the electric field at particle positions and the predicted particles by Euler scheme:
          \[
          E(x^{m}_i) = \sum\limits_{k=1}^{n_0} E^m(x_k) \hat{S}(x_k - x^{m}_i),\quad 
          x^{\text{guess}}_i = x^{m}_i + v^{m}_i \Delta t, \quad v^{\text{guess}}_i = v^{m}_i + E(x^{m}_i) \Delta t.
          \]
          \STATE Refine the electric field \(E^{\text{guess}}\), velocities \(v^{\text{guess}}_i\), and positions \(x^{\text{guess}}_i\)  over \(n\) iterations, using \eqref{V--A iteration} by replacing the values at $t^{m+1}$ on the right hand side with the previous guess.
          \STATE Update:
          \[
          E^{m+1} = E^{\text{guess}}, \quad x^{m+1}_i = x^{\text{guess}}_i \mod 2L, \quad v^{m+1}_i = v^{\text{guess}}_i, \ (i = 1,\cdots,N).
          \]
  \end{algorithmic}		
\end{algorithm}

We consider the Coulomb case (\(\gamma = -d = -2\)), where the collision kernel is defined as 
\[
A(z) = \frac{\Lambda}{|z|^2}\big(|z|^2 I_d - z \otimes z\big),
\]
with \(\Lambda = 0\) and \(\Lambda = 1\) representing the non-collision and strong collision cases, respectively. 
For simplicity, we assume that particles are uniformly distributed in the second spatial dimension, leading to homogeneity in this direction. 
As a result, the problem reduces to a scenario with one spatial dimension and two velocity dimensions (1D-2V).

The initial distribution is specified as a Maxwellian equilibrium with a spatial perturbation:
\[
f(x, v_x, v_y) = \frac{1 + \alpha \cos(0.5x)}{2\pi} \exp\Big(-\frac{v_x^2 + v_y^2}{2}\Big),
\]
where \(\alpha = 0.1, 0.5\). A larger \(\alpha\) corresponds to a stronger spatial perturbation.

The parameters are set as follows: \(t_0 = 0\), \(t_\text{end} = 50\), \(\Delta t = 0.02\), \(\Omega_x = [0, 4\pi]\), and \(\Omega_v = [-2\pi, 2\pi]^2\). 
For the spatial domain \(\Omega_x\), we use \(n_0 = 128\) grid points, while for each velocity dimension in \(\Omega_v\), we also use \(n_0 = 128\) grid points. Here $|\Omega_k|=\Delta x.$ 
The total number of particles is \(N = 500,000\), and \(n = 5\). 
Initial particle positions and velocities are sampled independently from the specified initial distribution.
    
For the shape function, we use \(S(x) = \hat{S}(x) / \Delta x\), where
\[
\hat{S}(x) =
\begin{cases}
\left(1 - \dfrac{|x|}{\Delta x}\right) & \text{if } |x| \leq \Delta x, \\
0 & \text{otherwise.}
\end{cases}
\]
For ~\eqref{charge density int} and ~\eqref{electric field int}, we can just take $\hat{S}(x)$ instead of $S(x)$ in these two equations \cite{DEROUILLAT2018351}, which are of the form:
\[
\begin{split}
&\rho(x_k)=\sum\limits_{i} \hat{S}(x_{i}-x_k)\frac{q}{|\Delta x|}-\rho_0,\\
& E(x_i)=\sum\limits_{k} E(x_k) \hat{S}(x_{k}-x_i).
\end{split}
\]

The detailed steps of the algorithm are provided in Algorithm~\ref{PIC+Landau collision}, which incorporates Algorithm~\ref{PIC solve electric field} for calculating the initial electric field and Algorithm~\ref{alg:CNimplicitsolver} for iterative solutions of the Vlasov--Amp\`ere equation. The results are presented in Figure~\ref{electric field L2 norm and total energy using spherical BM method}, showing the evolution of the electric field \(L^2\)-norm \(\|E\|_{L^2}\) and the total energy \(\mathcal{E}_{\text{total}}\) for different values of \(\alpha\) and \(\Lambda\). 
Our method, which combines the SBM scheme with the energy-conserving PIC algorithm, performs robustly across varying \(\alpha\) and \(\Lambda\). 
It preserves the conservation of the total energy \(\mathcal{E}_{\text{total}}\), as illustrated in Figure~\ref{total energy when alpha=0.1, alpha=0.5 1D2V Landau damping}. 
Furthermore, the damping behavior of the electric field \(E\) is clearly observed in Figures~\ref{electric field L2 norm when alpha=0.1 1D2V Landau damping} and \ref{electric field L2 norm when alpha=0.5 1D2V Landau damping}. For \(\Lambda=0\), \(\alpha=0.1 \text{ and } \alpha=0.5\) corresponds to linear Landau damping and strong nonlinear Landau damping, respectively. And the behavior of the electric field \(E\) is in good agreement with the numerical solutions in the previous researches \cite{FILBET2001166,NAKAMURA1999122}.

This example demonstrates that our SBM scheme can be effectively integrated with the energy-conserving PIC algorithm, providing an accurate and reliable solution for spatially non-homogeneous problems, such as the Vlasov--Poisson--Landau equation, while strictly maintaining key physical properties.

\begin{figure}[!t]
    \centering
        \subfigure[Electric field $L_2$ norm ($\alpha=0.1$)]{
        \centering
        \includegraphics[scale=0.22]{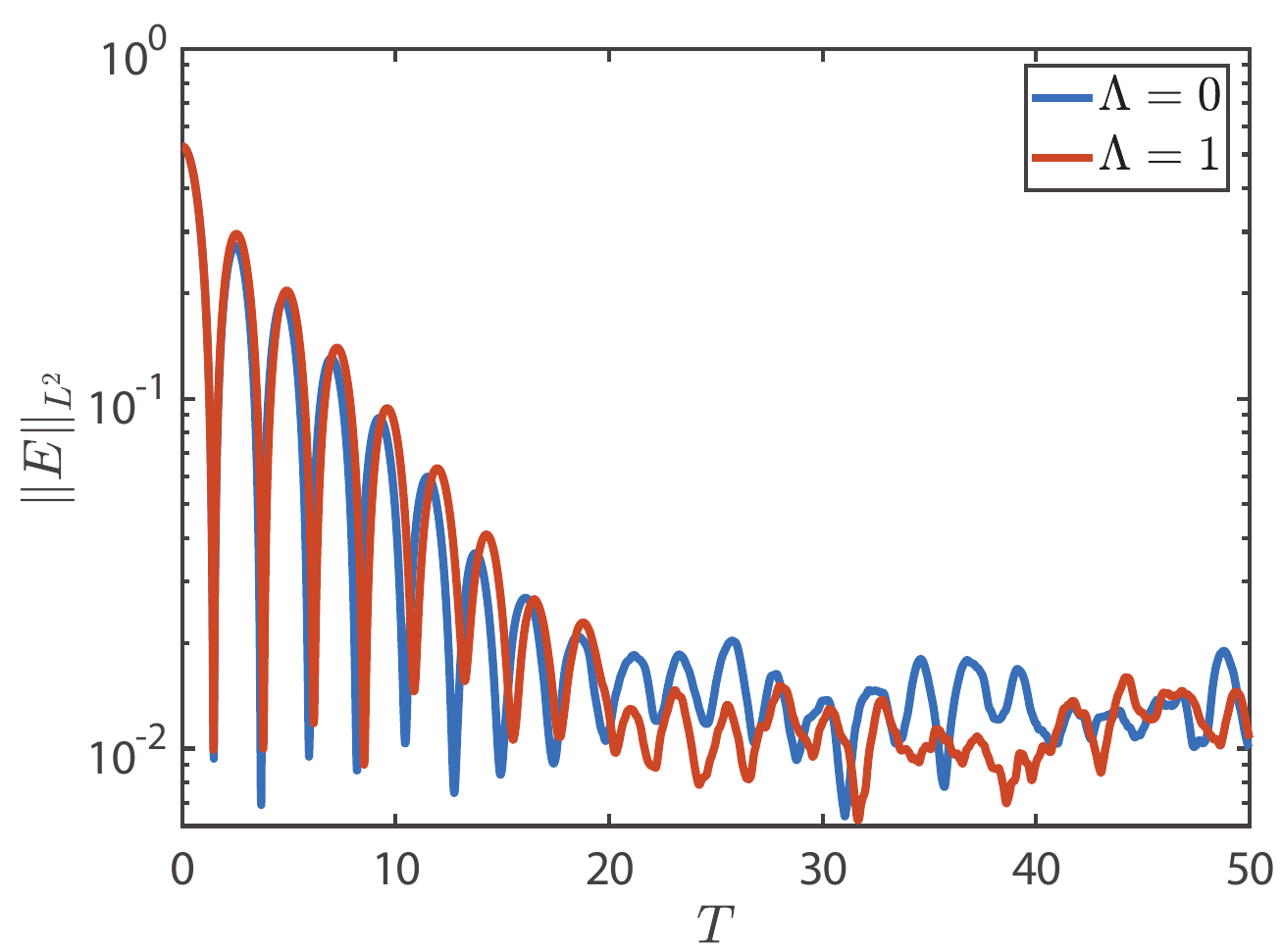}
        \label{electric field L2 norm when alpha=0.1 1D2V Landau damping}
}
        \centering
        \subfigure[Electric field $L_2$ norm ($\alpha=0.5$)]{
        \centering
        \includegraphics[scale=0.22]{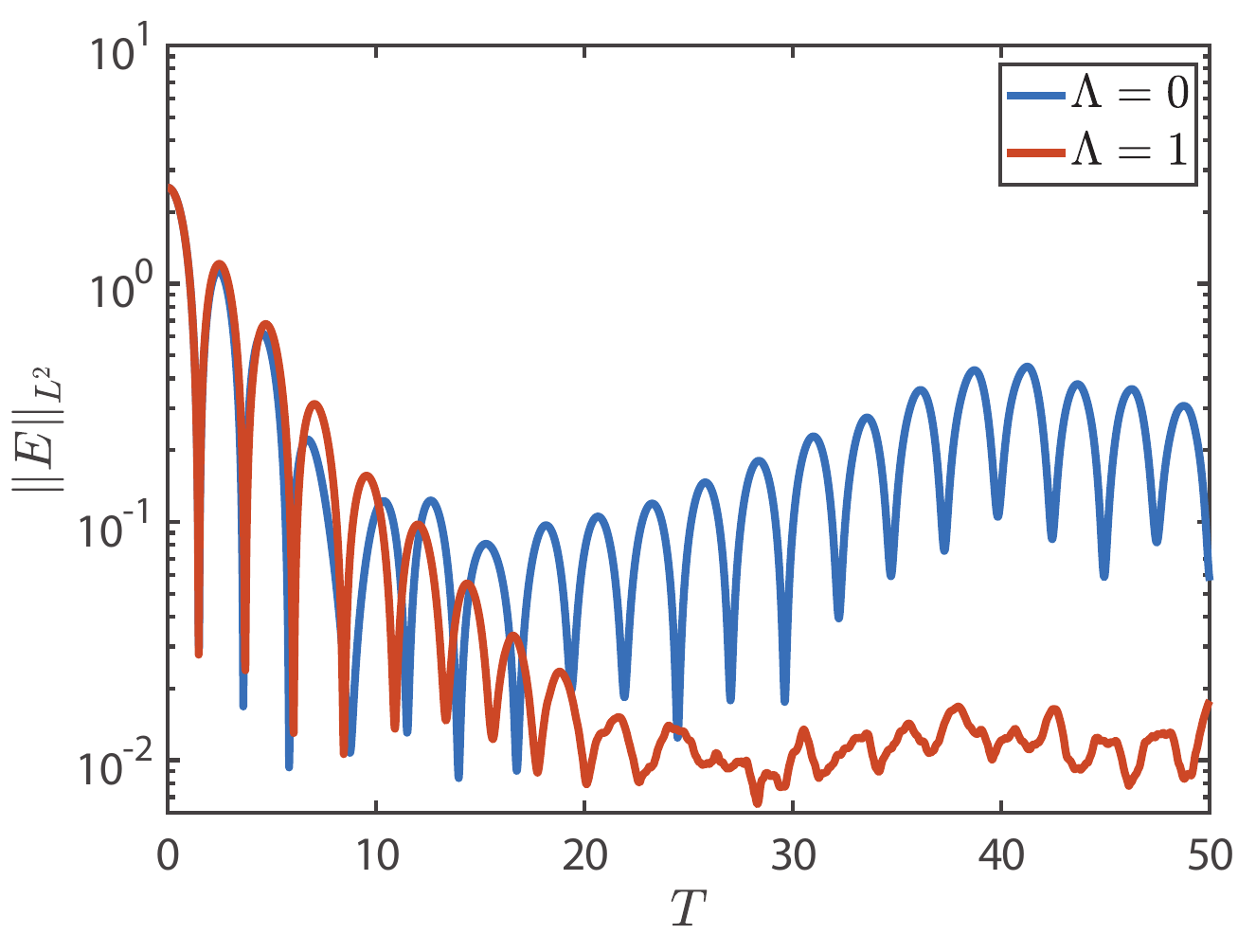}
        \label{electric field L2 norm when alpha=0.5 1D2V Landau damping}
}
        \centering
		\subfigure[Total energy ($\alpha=0.1,\,0.5$)]{
        \centering
        \includegraphics[scale=0.22]{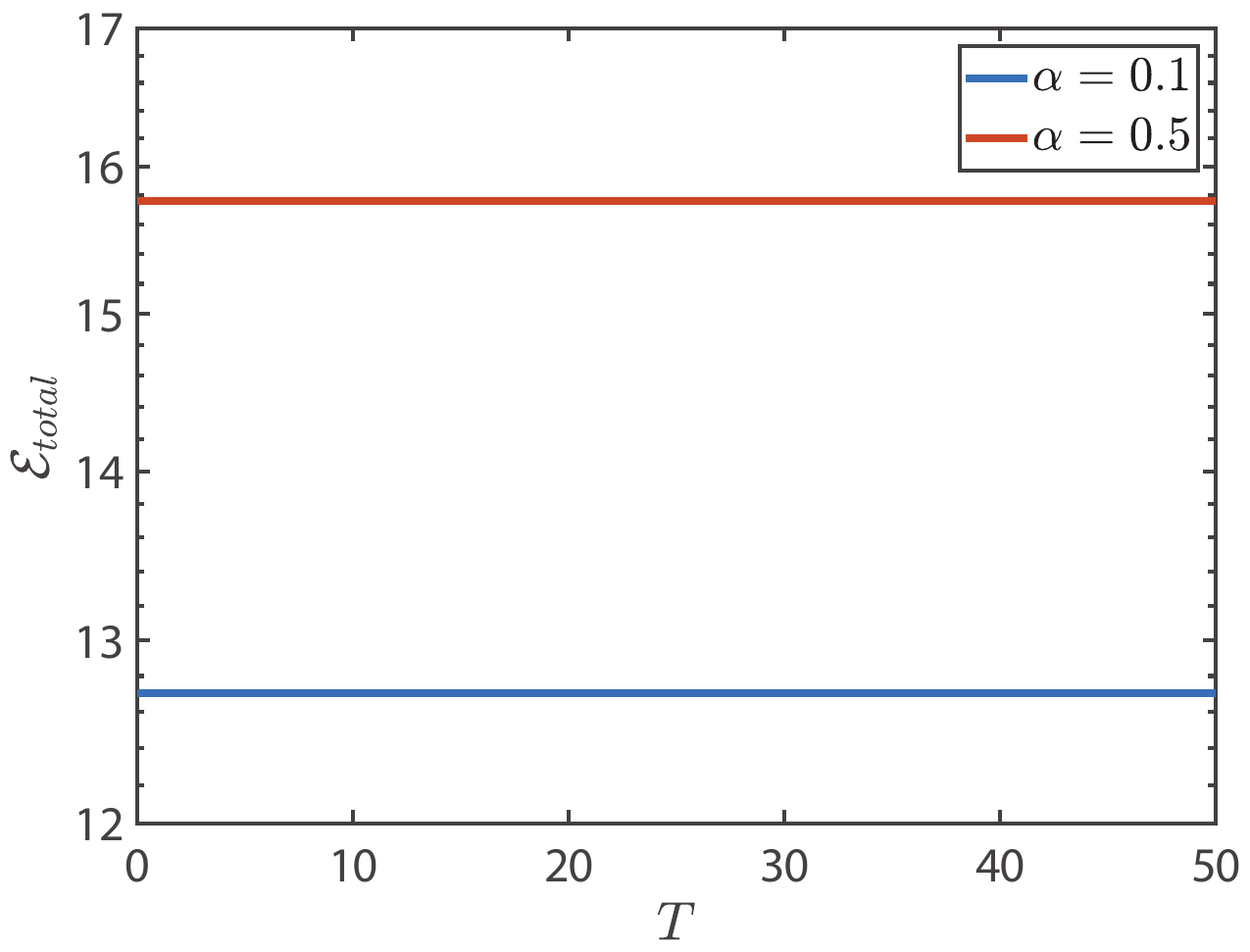}
        \label{total energy when alpha=0.1, alpha=0.5 1D2V Landau damping}
    }
    
	\centering
        \caption{(a,b) Electric field $L_2$ norm using SBM scheme for $\alpha=0.1, 0.5$; (c) Total energy using the SBM scheme. Note that the curves for $\Lambda=0$ and $\Lambda=1$ stacked together.}
        \label{electric field L2 norm and total energy using spherical BM method}
    \end{figure}

\section{Conclusion}

In this work, we introduced a stochastic particle method for the Landau kinetic equation that preserves key physical properties such as mass, momentum, energy conservation, and entropy dissipation. 
By modeling pairwise grazing collisions as diffusion processes, this method provides a simple yet effective way to simulate collisional plasmas. 
The exact temporal discretization ensures that the discrete system remains consistent with the continuous model, offering reliability in long-term simulations.

Numerical experiments have demonstrated accuracy and stability of our method across different scenarios, including Coulomb potentials and spatially non-homogeneous systems. 
Its computational efficiency, with $O(N)$ complexity per time step, makes it suitable for larger-scale applications. 
Nevertheless, as with any numerical method, the performance depends on the specific problem settings, and further investigation is needed to evaluate its limitations and applicability to more complex systems.

While the proposed method shows promise, there is still room for improvement and further exploration. 
For instance, analyzing its convergence under more general conditions and adapting it for problems with additional physical effects, such as external fields or anisotropic interactions, could extend its usefulness. 
Despite its current focus, the method offers a potential foundation for developing more robust tools in kinetic theory.

\section*{Acknowledgement}

This work was financially supported by the National Key R\&D Program of China, Project Number 2021YFA1002800.
The work of K. Du was supported by the National Natural Science Foundation of China (12222103).
The work of L. Li was partially supported by NSFC 12371400 and 12031013,  Shanghai Municipal Science and Technology Major Project 2021SHZDZX0102.

\bibliographystyle{plain}
\bibliography{landau}

\end{document}